\documentclass[11pt,letterpaper]{amsart}

\usepackage[letterpaper, left=1.75in, right=1.75in, top=1.4in, bottom=1.4in]{geometry}
\usepackage[utf8]{inputenc}
\usepackage[T1]{fontenc}
\usepackage{paralist}
\usepackage{amssymb}
\usepackage{amsmath}
\usepackage{amsthm}
\usepackage[english]{babel}
\usepackage{chngcntr}
\usepackage{fancyhdr}
\usepackage{mathrsfs}
\usepackage[basic]{mnotes}
\usepackage{mdwlist}

\usepackage{graphicx}
\usepackage{color,xcolor}
\usepackage{mathtools}

\numberwithin{equation}{section}
\numberwithin{figure}{section}

\theoremstyle{plain}
\newtheorem{thm}{Theorem}[section]

\newtheorem{lem}[thm]{Lemma}
\newtheorem{cor}[thm]{Corollary}  
\newtheorem{prop}[thm]{Proposition}

\newcounter{introthm}

\newtheorem{introtheorem}[introthm]{Theorem}

\newtheorem{introcoro}[introthm]{Corollary}

\theoremstyle{definition}
\newtheorem{defn}[thm]{Definition}

\theoremstyle{remark}
\newtheorem{rem}[thm]{Remark}

\newtheoremstyle{myrem}{3pt}{3pt}{\normalfont}{0pt}{\bfseries}{.}{0.5em}{}
\theoremstyle{myrem}

\newcommand{\mc}{\mathcal}
\newcommand{\w}{\omega}

\renewcommand{\:}{\colon}
\renewcommand{\d}{\partial}

\newcommand{\g}{R}

\newcommand{\V}{\hat V}
\newcommand{\E}{\hat E}
\newcommand{\B}{\hat B}
\newcommand{\q}{q}

\newcommand{\e}{e}

\newcommand{\aaa}{$1$-\resizebox{!}{2.0ex}{$\stackrel{2}{\textrm{to}}$}-$1$}
\newcommand{\aab}{$1$-\resizebox{!}{2.0ex}{$\stackrel{3}{\textrm{to}}$}-$1$}
\newcommand{\aac}{$1$-\resizebox{!}{2.0ex}{$\stackrel{n}{\textrm{to}}$}-$1$}
\newcommand{\SC}{S_{C}}
\newcommand{\SCP}{S_{C'}}
\newcommand{\SW}{S_{\partial W}}
\newcommand{\SWP}{S_{\partial W'}}
\newcommand{\SWH}{S_{\partial \hat W}}
\newcommand{\SWPI}{S_{\pi(h(D))}}
\newcommand{\SD}{S_D}

\newcommand{\s}{\sigma}
\renewcommand{\S}{\Sigma}

\newcommand{\thin}{C^{\ast}}
\newcommand{\tthin}{C^{\ast\ast}}

\newcommand{\Thin}{\textrm{lev}}
\renewcommand{\g}{g}

\renewcommand{\limsup}{\varlimsup}
\renewcommand{\liminf}{\varliminf}

\newcommand{\ssq}{\ensuremath{\subseteq}}

\newcommand{\eps}{\ensuremath{\varepsilon}}

\newcommand{\alphlist}{\begin{list}{(\alph{enumi})}{\usecounter{enumi}\setlength{\parsep}{2pt}
      \setlength{\itemsep}{1pt} \setlength{\topsep}{5pt}
      \setlength{\partopsep}{3pt}}}
\newcommand{\arablist}{\begin{list}{(\arabic{enumi})}{\usecounter{enumi}\setlength{\parsep}{2pt}
          \setlength{\itemsep}{1pt} \setlength{\topsep}{5pt}
          \setlength{\partopsep}{3pt}}}
\newcommand{\romanlist}{\begin{list}{(\roman{enumi})}{\usecounter{enumi}\setlength{\parsep}{2pt}
              \setlength{\itemsep}{1pt} \setlength{\topsep}{5pt}
              \setlength{\partopsep}{3pt}}}
\newcommand{\Romanlist}{\begin{list}{(\Roman{enumi})}{\usecounter{enumi}\setlength{\parsep}{2pt}
              \setlength{\itemsep}{1pt} \setlength{\topsep}{5pt}
              \setlength{\partopsep}{3pt}}}
\newcommand{\bulletlist}{\begin{list}{$\bullet$}{\setlength{\parsep}{2pt}
                \setlength{\itemsep}{1pt} \setlength{\topsep}{5pt}
                \setlength{\partopsep}{3pt}\setlength{\leftmargin}{15pt}}} 
\newcommand{\Alphlist}{\begin{list}{(\Alph{enumi})}{\usecounter{enumi}\setlength{\parsep}{2pt}
      \setlength{\itemsep}{1pt} \setlength{\topsep}{5pt}
      \setlength{\partopsep}{3pt}}}
 \newcommand{\listend}{\end{list}}

\newcommand{\T}{\ensuremath{\mathbb{T}}}

\newcommand{\N}{\ensuremath{\mathbb{N}}} 
\newcommand{\R}{\ensuremath{\mathbb{R}}}
\newcommand{\Z}{\ensuremath{\mathbb{Z}}}

\title[Birkhoff Spectra of almost one-to-one extensions]{Birkhoff Spectra of symbolic almost one-to-one extensions} 
\author{Gabriel Fuhrmann}
\address{Department of Mathematical Sciences, Durham University, UK}
\email{gabriel.fuhrmann@durham.ac.uk}

\newcounter{counter}

\begin{document}

\makeatletter
\typeout{Text width: \the\textwidth}
\typeout{Text height: \the\textheight}
\makeatother

\begin{abstract}
Given a continuous self-map $f$ on some compact metrisable space $X$,
it is natural to ask for the visiting frequencies of points $x\in X$
to sufficiently ``nice'' sets $C\ssq X$ under iteration of $f$.

For example, if $f$ is an irrational rotation on the circle,
it is well-known that the Birkhoff average $\lim_{n\to\infty}1/n\cdot \sum_{i=0}^{n-1}\mathbf 1_C(f^i(x))$ exists and equals $\textrm{Leb}_{\T^1}(C)$ for all $x$ whenever $C$ is measurable with boundary $\d C$ of zero Lebesgue measure.
If, however, $\d C$ is fat (of positive measure), the respective averages can generally only be evaluated almost everywhere or on residual sets.
In fact, there does not appear to be a single example of a fat Cantor set $C$ whose \emph{Birkhoff spectrum}---the full set of visiting frequencies---is known.

In this article, we develop an approach to analyse the Birkhoff spectra of a natural class of dynamically defined fat nowhere dense compact subsets of Cantor minimal systems.
We show that every Cantor minimal system admits such sets whose Birkhoff spectrum is a full non-degenerate interval---and also such sets for which the spectrum is not an interval. 
As an application, we obtain that every irrational rotation admits fat Cantor sets $C$ and $C'$ whose Birkhoff spectra are, respectively, an interval and not an interval.
\end{abstract}

\maketitle
\section{Introduction}\label{sec: introduction}
Given an irrational rotation $(\T^1,w)$ on the circle, Baire's Category Theorem implies that if $C\ssq \T^1$ is a Cantor set, then there are residually many $\theta\in \mathbb T^1$ that avoid $C$, in particular,
\[\lim_{n\to \infty} S^n_C(\theta)=\lim_{n\to \infty}\frac 1n \cdot\sum_{i=0}^{n-1} \mathbf 1_C(\theta+i\w)=0.\]
On the other hand, if $C$ is fat (that is, $\textrm{Leb}_{\mathbb T^1}(C)>0$),
Birkhoff's Ergodic Theorem gives
\[\lim_{n\to\infty} S^n_C(\theta)=\textrm{Leb}_{\mathbb T^1}(C)>0\] for $\textrm{Leb}_{\mathbb T^1}$-a.e.\ $\theta$.
Moreover, as a simple consequence of the unique ergodicity of $(\T^1,\w)$, we have $\limsup_{n\to\infty} S^n_C(\theta)\leq \textrm{Leb}_{\mathbb T^1}(C)$ for each $\theta \in \T^1$.
Besides these well-known and classical observations, however,
the range of possible visiting frequencies to $C$ remains poorly understood.

Some years ago, 
Kwietniak asked if, besides the above, there is anything else we can say about the Birkhoff spectrum, that is, the collection of all accumulation points
\[
S_C=\bigcup_{\theta\in \T^1}\bigcap_{N\in \N}\overline{\left\{S^n_C(\theta)\: n\geq N\right\}}
\]
when $C$ is a fat Cantor set \cite{KwietniakOverflow}.
Specifically, is it possible that $S_C=[0,\textrm{Leb}_{\T^1}(C)]$?
As we will show, the answer is \emph{yes}, see Corollary~\ref{cor: birkhoff spectra of irrational rotations introduction}.

Questions like the above naturally arise when dealing with irregular or, more generally, weak model sets, where fat Cantor sets often appear as (boundaries of) the respective windows.
In part due to their strong links to $\mc B$-free systems and Sarnak’s M\"obius disjointness programme,
there is a recent surge of interest in this area; see \cite{BaakeJagerLenz2016,JaegerLenzOertel2016PositiveEntropyModelSets,KasjanKellerLemanczyk2019,FuhrmannGlasnerJagerOertel2021,zbMATH07990784} and references therein.
While our work follows a somewhat independent direction in addressing the above question, we develop an approach that allows for a detailed analysis of the possible ergodic averages and thus contributes to the broader efforts in the field. 
Moreover, due to the flexibility of our method, we expect it to be applicable to a range of other problems.

\subsection{Main results}
Naively and obviously flawed, one might try proving $S_C=[0,\textrm{Leb}_{\T^1}(C)]$ by realising any frequency in $[0,\textrm{Leb}_{\T^1}(C)]$ through orbits that alternately shadow---for suitably chosen stretches of time---trajectories which never hit $C$ and trajectories which visit $C$ with a frequency $\textrm{Leb}_{\T^1}(C)$; after all, by minimality of irrational rotations, each point comes arbitrarily close to any other point.
As flawed as this strategy is, it actually works in a different setting and under certain assumptions---specifically, the setting addressed in our first main result.
\begin{introtheorem}\label{thm: A}
    Suppose $(\hat\S,\s)$ and $(\S,\s)$ are minimal subshifts with a topological factor map  $\pi\:(\hat\S,\s)\to (\S,\s)$ such that
    \begin{itemize}
        \item $\pi$ is almost $1$-to-$1$, that is, there is $x\in \S$ with $|\pi^{-1}(x)|=1$;
        \item Each fibre has at most $2$ elements, that is, $|\pi^{-1}(x)|\leq 2$ ($x\in \S$).
    \end{itemize}    
    Then $S_D=[0,\sup_\mu \mu(D)]$, where the supremum is over all invariant measures $\mu$ of $(\S,\s)$.
\end{introtheorem}
Here, $D$ is the collection of points $x=(x_n)_{n\in\Z}\in \S$ such that there are $(y_n)_{n\in\Z},(z_n)_{n\in\Z}\in \pi^{-1}(x)$ with $y_0\neq z_0$; $S_D$ comprises all possible frequencies of visits to $D$ analogously to $S_C$.\footnote{Formally, 
$S_D=\bigcup_{x\in \S}\bigcap_{N\in \N}\overline{\left\{S^n_D(x)\: n\geq N\right\}}$,
where $S^n_D(x)= 1/n \cdot\sum_{i=0}^{n-1} \mathbf 1_D(\s^i x)$ and $D$ is as above.
}
We emphasise that, in the above statement, the symbolic setting is not only more tractable from a combinatorial perspective than analogous problems for irrational rotations; owing to an array of available codings across different classes of systems, 
 it also lends itself to a broader range of potential applications.
We will take advantage of this fact in Section~\ref{sec: last section}.


It is also important to note that we show---through a simple, explicit construction---that any minimal subshift $(\Sigma,\sigma)$ admits an extension $(\hat \S,\s)$
satisfying the assumptions of Theorem~\ref{thm: A} with $\mu(D)>0$ for every invariant measure $\mu$, see Theorem~\ref{thm: 2 to 1 extensions exist}.
En passant, this provides an alternative to the constructions of irregular zero entropy model sets in \cite{BaakeJagerLenz2016,FuhrmannGlasnerJagerOertel2021}, see also the discussion in Remark~\ref{rem: entropy conjecture model sets}.

Our second main result establishes that in Theorem~\ref{thm: A}, the assumption of an upper bound of $2$ on the fibre cardinality is optimal in some sense.
\begin{introtheorem}\label{thm: B}
   Each minimal subshift $(\S,\s)$    
   admits a minimal almost $1$-to-$1$ extension with at most $3$ elements in each fibre, and such that $\{0\}\subsetneq S_D\neq [0,\sup_\mu \mu(D)]$.
\end{introtheorem}
Interestingly, the examples obtained in our proof of the above theorem exhibit a spectral gap at $0$ but not at $\sup_\mu \mu(D)$.
Whether it is possible to achieve $S_D= \{0\} \cup \{\mu(D)\: \mu \text{ invariant}\}$ (with $\sup_\mu \mu(D)>0$) remains open.


Finally, we apply our main results to codings of minimal rotations on compact monothetic groups.
In the special case of rotations on the circle, we obtain 
\begin{introcoro}\label{cor: birkhoff spectra of irrational rotations introduction}
Given an irrational rotation $(\T^1,\w)$, there exist fat Cantor sets $C,C'\ssq\T^1$ such that $S_C=[0,\textrm{Leb}_{\T^1}(C)]$, and such that $\SCP$ has a gap at $0$; in particular, $\{0,\textrm{Leb}_{\T^1}(C')\}\ssq \SCP \subsetneq[0,\textrm{Leb}_{\T^1}(C')]$.
\end{introcoro}

\subsection{Outline}
This article is organised as follows.
Terminology and background of those concepts we use all through the article are discussed in the next section.
Key to our analysis is the well-established machinery around Bratteli-Vershik representations
of Cantor minimal systems.
An important feature of these representations is that they provide us with a very explicit and straightforward characterisation of almost $1$-to-$1$ extensions---they are obtained through what we call \emph{copy-pasting}.
The basics and some basic consequences of this characterisation are discussed in Section~\ref{sec: almost 1 to 1 extensions...}.
In Section~\ref{sec: proof of the main theorem}, we prove a slightly more general version of Theorem~\ref{thm: A}, see Theorem~\ref{thm: main}.
A similarly more general version of Theorem~\ref{thm: B} is proven in  Section~\ref{sec: non-maximal entire section}, see Theorem~\ref{thm: non-maximal spectra exist}.
In the last part, Section~\ref{sec: last section}, we translate our main results to statements on visiting frequencies to certain kinds of nowhere dense sets on compact monothetic groups, including fat Cantor sets on $\T^1$.
This translation utilises almost automorphic subshifts; all the additional background we need in that context is discussed at the beginning of Section~\ref{sec: last section}.

\section{Preliminaries}\label{sec: preliminaries}
This section introduces the notation used throughout the paper and briefly reviews key concepts, in particular Bratteli-Vershik systems. 
While we do not aim for a comprehensive exposition, the essential concepts are discussed sufficiently for the purposes of this work. 
For further background, we refer the reader to the literature, see e.g.\ \cite{MarkleyPaul1979,Walters1982,auslander1988,HermanPutnamSkau1992}.
\subsection{Basic notions from topological dynamics}
A \emph{(topological) dynamical system} is
a pair $(X,f)$ where $X$ is a compact metrisable space and $f\:X\to X$ is a homeomorphism on $X$.
Given two dynamical systems $(X,f)$ and $(Y,g)$, a continuous onto map $\pi\: X\to Y$ is a \emph{factor map} if $\pi\circ f=g\circ \pi$---we may write $\pi\:(X,f)\to(Y,g)$.
In this case, we call $(X,f)$ an \emph{extension} of $(Y,g)$ and the latter a \emph{factor} of the former.
If $\pi$ is a homeomorphism, we call it an \emph{isomorphism} and say that $(X,f)$ and $(Y,g)$ are \emph{isomorphic}.

Given $\pi\:(X,f)\to(Y,g)$ and $y\in Y$, we call $\pi^{-1}y$ the \emph{fibre} (or \emph{$\pi$-fibre}) \emph{over} $y$---note that here as well as at various other places in this article, we avoid explicit bracketing.
We say a fibre is \emph{regular} if it is a singleton and otherwise, we say it is \emph{irregular}.
Identifying regular fibres with their unique element,
we say that $\pi$ is an \emph{almost $1$-to-$1$} factor map if the set
of regular fibres is dense in $X$.
If $(X,f)$ is minimal, this is equivalent to having
at least one regular fibre.
Here, $(X,f)$ is \emph{minimal} if for each $x\in X$ its \emph{orbit} $\mc O(x)=\{f^\ell(x)\:\ell \in \Z\}$ is dense in $X$.
A \emph{Cantor} minimal system is a minimal system $(X,f)$ where $X$ is a Cantor set.

Given a topological dynamical system $(X,f)$ and a Borel probability measure $\mu$ on $X$, we call
$\mu$ an \emph{invariant measure} if $\mu(A)=\mu(fA)$ for each Borel set $A$.

An almost $1$-to-$1$ factor map $\pi\:(X,f)\to(Y,g)$ is  \emph{regular} if
for every invariant measure $\mu$ of $(Y,g)$,
the fibre over $\mu$-almost every point is regular.
At the other extreme, if 
the projection of regular fibres is $\mu$-null for each invariant measure $\mu$, we call $\pi$ \emph{irregular}.
In either case, if the map $\pi$ is clear from the context or irrelevant, we just say that $(X,f)$ is a regular (an irregular) almost $1$-to-$1$ extension of $(Y,g)$.
Note that, regardless of (ir-)regularity, whenever $(X,f)$ is an almost $1$-to-$1$ extension of $(Y,g)$, the regular fibres are residual in $X$ and their projection (under the corresponding factor map) is residual in $Y$.

We are particularly interested in irregular almost $1$-to-$1$ extensions where there is a uniform upper bound
$n\in \N$ on the cardinality of each fibre.
For brevity, we may refer to such extensions (factor maps) as
irregular almost \emph{\aac}\ extensions (factor maps).

\subsection{Background on Bratteli diagrams}
 A \emph{Bratteli diagram} $B=(V,E)$ is an infinite graph with \emph{vertices} $V$ and \emph{edges} $E$ such that
 \begin{enumerate}
  \item $V=\bigsqcup_{n\in \Z_{\geq 0}} V_n$ with $V_0=\{v_0\}$ and $1\leq | V_n|<\infty$ ($n\geq 1$);
  \item $E_n=\bigsqcup_{n\in \Z_{\geq 0}} E_n$ with $1\leq | E_n|< \infty$ ($n\in \Z_{\geq 0}$);
  \item There are maps $r,s\: E \to V$ with $r(E_n)=V_{n+1}$ and $s(E_n)=V_n$ for all $n\geq 0$.
 \end{enumerate}
 Given $e\in E$, we call $r(e)$ the \emph{range} of $e$ and $s(e)$ its \emph{source}.
 A (possibly finite) sequence $(e_0,e_1,\ldots)$ in $E$ is a \emph{path} if it satisfies $r(e_{i-1})=s(e_i)$ for each $i\geq 1$.
 If $\gamma=(e_0,e_1,\ldots, e_n,\ldots)$ is a path with $r(e_n)=v\in V$, we say
 $\gamma$ \emph{traverses} $v$.
 We extend the maps $r$ and $s$ to paths by setting $s(e_0,e_1,\ldots)=s(e_0)$ (the \emph{source} of $(e_0,e_1,\ldots)$) and
 $r(e_0,e_1,\ldots,e_n)=r(e_n)$ (the \emph{range} of $(e_0,e_1,\ldots,e_n)$).
 
 Given $0\leq n<N$, we write $E_{n,N}$
 for the collection of all finite paths $\gamma$ with source in $V_n$ and
 range in $V_N$; for $v\in V_n$ and $v'\in V_N$, we write
 $E(v,v')$ for the collection of all paths $\gamma$ with $s(\gamma)=v$ and $r(\gamma)=v'$.
 The collection of all infinite paths with source $v_0\in V_0$ is denoted by $X_B$.
 Given a path $(e_0,e_1,\ldots,e_n,\ldots )$, we call $(e_0,e_1,\ldots,e_n)$ its
 \emph{$n$-head} ($n\geq 0$).
 For a finite path $\gamma=(e_0,e_1,\ldots,e_n)$ with $s(\gamma)=v_0$, we define
 $[\gamma]=[e_0,e_1,\ldots,\e_n]$ to be the collection all paths in $X_B$ whose
 $n$-head coincides with $\gamma$ and call $[e_0,e_1,\ldots,\e_n]$ an \emph{$n$-cylinder}, where $n\geq 0$.
 We equip $X_B$ with the topology generated by the collection of all $n$-cylinders.
 Unless stated otherwise, we only consider such diagrams where $X_B$ is a Cantor set with this topology.

 Given an infinite sequence $n_0=0<n_1<n_2<\ldots$ in $\Z_{\geq 0}$,
 we may \emph{telescope $B$ (along $(n_k)_{k\geq 0}$) to a diagram $B'=(V',E')$}
 where $V_k'=V_{n_k}$ and $E'_k=E_{n_k,n_{k+1}}$ for each $k\geq 0$ and $r$ and $s$ (the range and source maps of $B'$) are defined in the obvious way.
 If $n_{i+1}=n_i+1$ for all $i\geq 1$, we may simply say that $B'$ is obtained by telescoping $B$ \emph{between level $0$ and level $n_1$}.
 In this situation, if, conversely, we want to understand $B$ as being obtained from $B'$, we may also say that we get $B$ from $B'$ by \emph{introducing} new levels between level $0$ and level $1$.
 Unless mentioned otherwise, we only consider \emph{simple} Bratteli diagrams $B$ which, by definition,
 can be telescoped to some diagram $B'$ where for all $n\in \Z_{\geq 0}$ and each $v\in V_n'$ and $w\in V_{n+1}'$ there is $e\in E_n'$ with $s(e)=v$ and $r(e)=w$.

 Given a Bratteli diagram $(V,E)$ and a partial order $\leq $ on $E$, we call $B=(V,E,\leq)$ an \emph{ordered Bratteli diagram} if $e$ and $e'$ are comparable (that is, $e\leq e'$ or $e'\leq e$) if and only if $r(e)=r(e')$. 
 See Figure~\ref{fig: general sturmian bv diagram} for an example of an ordered Bratteli diagram.
 The order $\leq$ extends in the obvious way to finite paths, where $\gamma$ and $\gamma'$ are comparable if and only if they \emph{start} on the same level (that is, $s(\gamma),s(\gamma')\in V_k$) and \emph{end} in the same vertex (that is, $r(\gamma)=r(\gamma')$).
 We further extend $\leq$ to $X_B$ where $\gamma=(e_0,e_1,\ldots)\leq \gamma'=(e_0',e_1',\ldots)$  if and only if there is $n\in \Z_{\geq 0}$ such that $e_N=e'_N$ for all $N> n$ and the respective $n$-heads satisfy $(e_0,\ldots, e_n)\leq (e'_0,\ldots, e'_n)$.
 A diagram obtained by telescoping an ordered Bratteli diagram $B$ inherits the order from $B$ in the obvious way.
 \begin{figure}[ht]
\begin{center}
\includegraphics[scale=0.97, alt={A graph showing vertices organised by levels $V_0$ through $V_3$. 
Level $V_0$ contains a single vertex. 
Each subsequent level $V_n$ for $n \geq 1$ contains two vertices, labelled $0$ and $1$. 
At level $V_1$, each vertex is connected to the vertex in $V_0$ by exactly one edge.
At level $V_2$, vertex $1$ in $V_1$ connects to vertices in $V_2$ only through a minimal edge, while all other edges (with $n_1+1$ edges for vertex $2$ and $n_1$ edges for vertex $1$) originate from vertex $0$ in $V_1$. 
The structure at level $V_3$ is similar.}]{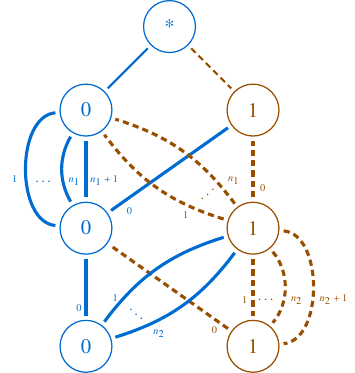}
 \caption{First levels of an ordered Bratteli diagram representing a Sturmian subshift, see \cite{MorseHedlund1940,DartnellDurandMaass2000}.
 The labels of the edges indicate the order.}
 \label{fig: general sturmian bv diagram}
\end{center}
\end{figure}

 We call a path \emph{maximal} (\emph{minimal}) if all its edges are maximal (minimal).
If a path $\gamma$ which starts in $v_0$ is not maximal, it has a unique successor in $\leq$ which we denote by $\phi_B(\gamma)$ (or simply $\phi(\gamma)$ if $B$ is clear from the context).
We always assume that $B$ is \emph{properly ordered}, that is, $X_B$ has a unique maximal path $\gamma^+$ and a unique minimal path $\gamma^-$.
Setting $\phi_B(\gamma^+)=\gamma^-$, $\phi_B$---when seen as a self-map on $X_B$---becomes a homeomorphism on $X_B$, which is referred to as the \emph{Vershik map} of $B$.
Note that as $B$ is properly ordered, we can enforce by telescoping that the minimal (maximal) edges on each level start in a unique vertex of the respective previous level.

It is well-known and not hard to see that the \emph{Bratteli-Vershik system}
$(X_B,\phi_B)$ is a Cantor minimal system.
The converse---that is, the fact that every Cantor minimal system has a representation as a Bratteli-Vershik system---is also well-known but less straightforward.
Specifically, given a Cantor minimal system $(X,f)$ and some point $x\in X$, there is a simple properly ordered diagram $B$ and an isomorphism $h\:(X_B,\phi_B)\to (X,f)$ which sends the unique minimal path to $x$, see e.g.\ \cite{HermanPutnamSkau1992} for the details.
We call $B$ as well as the pair $(B,h)$ or the associated system $(X_B,\phi_B)$ a \emph{Bratteli-Vershik representation} of $(X,f)$.
Note that if $B$ is a Bratteli-Vershik representation of $(X,f)$, then so is any diagram obtained through telescoping $B$.

\subsection{Invariant measures of Bratteli-Vershik systems}
We only need a very limited understanding of invariant measures of Bratteli-Vershik systems, see e.g. \cite{BezuglyiEtAl2010} for more background.

Consider an ordered Bratteli diagram  $B=(V,E,\leq)$ satisfying the assumptions from above.
In particular, $X_B$ is properly ordered and gives hence rise to a Bratteli-Vershik system $(X_B,\phi_B)$.
For $v\in V_n\ssq V$ with $n\geq 1$, we set
\[X_v=\{(e_n)_{n\geq 0}\in X_B\:(e_n)_{n\geq 0} \text{ traverses } v\}.\]
Given $v\in V_n$ ($n\geq 1$) and an invariant measure $\mu$ of $(X_B,\phi_B)$, it is straightforward to see that for all finite paths $\gamma,\gamma'\in E(v_0,v)$, we have 
\begin{align}\label{eq: inv measure}
\mu([\gamma])=\mu([\gamma'])\qquad \text{and hence} \qquad \mu([\gamma])/\mu(X_v)=1/|E(v_0,v)|.
\end{align}

\section{Almost one-to-one extensions, extension triples and extended Bratteli diagrams}\label{sec: almost 1 to 1 extensions...}
The last decade has seen a number of important structural results on factor relations between Cantor minimal systems via their Bratteli-Vershik representations \cite{Sugisaki2011,GlasnerHost2013,AminiElliotGolestani2021,GolestaniHosseini2021}. 
We do not make direct use of these works, but instead develop what we need alongside introducing our own notation. 
One exception is \cite{Sugisaki2011}, which implicitly contains---in slightly different notation---the representation of almost one-to-one factor maps in Theorem~\ref{thm Sugisaki}.

In the following two subsections, we mainly introduce a convenient terminology and notation for dealing with almost $1$-to-$1$ Cantor extensions.
Based on these preparations, we then introduce \emph{extended} Bratteli diagrams in Section~\ref{sec: extended BV}.
As a first application, we show the existence of irregular almost \aaa\ extensions for any Cantor minimal system, Theorem~\ref{thm: 2 to 1 extensions exist}.

\subsection{Notation}\label{sec: notation}
Given an ordered Bratteli diagram $B=(V,E,\leq)$, for each $n\geq 0$, we write $V_n=\{v_0(n),\ldots,v_{|V_n|-1}(n)\}$ and
\[
    E_n=\{e_{\ell,m}(n)\:\ell=0,\ldots, |V_{n+1}|-1 \text{ and } m=0,\ldots, r_\ell(n)-1\}.
\] 
Here,
\begin{align}\label{eq: defn rell}
r_\ell(n)=|\{e\in E_{n}\:r(e)=v_\ell(n+1)\}| 
\end{align}
and $(e_{\ell,m})_{m=0,\ldots,r_\ell(n)-1}$ is the family of all edges which end in $v_\ell(n+1)$.
We assume that for all $n\geq0$, $e_{\ell,M}(n)>e_{\ell,m}(n)$ if $M>m$.
If there is no risk of ambiguity, we may suppress the dependence on the level $n$ in any of the notation introduced so far.

Given two properly ordered Bratteli diagrams $B=(V,E,\leq)$ and $\hat B=(\hat E,\hat V,\hat\leq)$, we say $\B$ is obtained by \emph{copy-pasting $B$} if for each vertex $v_\ell \in V_n$ (with $n\geq 1$), there is at least one copy in $\V_n$ which is connected---via the edges from $\hat E_{n-1}$---to the vertices of $\hat V_{n-1}$ in a similar way the vertex $v_\ell$ is connected---via the edges from $E_{n-1}$---to the vertices in $V_{n-1}$.
In formal terms, $\B$ is obtained by copy-pasting $B$ if---possibly after relabelling---we have $\V_0=\{v_0\}$ and
\begin{enumerate}[(i)]
 \item For each $n\geq 1$ and $\ell=0,\ldots,|V_n|-1$, there is $j_\ell(n)\in \N$ such that \label{item: copies of vertices}
 \[\hat V_n=\{v_\ell^{j}(n)\:0\leq\ell<|V_n|,\, j=0,\ldots,j_\ell(n)-1\};\]
 \item For each $n\geq 0$ and with $r_\ell(n)$ as in \eqref{eq: defn rell}, \label{item: copies of egdes}
 \[\hat E_{n}=\{e_{\ell,m}^{j}(n)\:0\leq\ell< |V_{n+1}|,\, 0\leq m< r_\ell(n),\, 0\leq j< j_\ell(n+1)\},\]
 where $r(e_{\ell,m}^j(n))=v_{\ell}^{j}(n+1)$ for all admissible $\ell$, $m$, and $j$;
 \item For each $n\geq1$ and all admissible $\ell,m,j$, there is $j' \in \{0,\ldots,j_{\ell'}(n)-1\}$ with
 $s(e_{\ell,m}^j(n))=v_{\ell'}^{j'}(n)$, where $\ell'$ satisfies
 $s(e_{\ell,m}(n))=v_{\ell'}(n)$ (in $B$);\label{item: copied edges are connected like their originals} 
 \item The order $\hat \leq $ on $\E$ is inherited from the order $\leq$ on $E$ in the obvious way, that is, $e_{\ell,m}^j(n)\,{\hat <}\,e_{\ell',m'}^{j'}(n')$ if
and only if $j=j'$ and $e_{\ell,m}(n)<e_{\ell',m'}(n')$ (in $B$).
\end{enumerate}
If $\B$ is obtained by copy-pasting $B$, we call the vertices $v_\ell^j(n)$ and edges $e_{\ell,m}^j(n)$ \emph{copies} of $v_\ell(n)$ and $e_{\ell,m}(n)$, respectively.
We may also call $v_0\in \V_0$ a copy of $v_0\in V_0$.
With this terminology, we can rephrase \eqref{item: copied edges are connected like their originals} by saying that if
$e_{\ell,m}(n)$ is an edge that starts in a vertex
$v_{\ell'}(n)$, then each copy of
$e_{\ell,m}(n)$ starts in a copy of $v_{\ell'}(n)$.
See Figure~\ref{fig: extension of sturmian} for an example of a diagram that is obtained via copy-pasting the diagram from Figure~\ref{fig: general sturmian bv diagram}.
 \begin{figure}[ht]
\begin{center}
 \includegraphics[scale=0.97, alt={A graph showing vertices arranged by levels $V_0$ through $V_3$, corresponding to those in Figure~\ref{fig: general sturmian bv diagram}. 
 Level $V_0$ contains a single vertex. 
 Level $V_1$ has three vertices: two are copies of vertex $0$ from diagram $B$ in Figure~\ref{fig: general sturmian bv diagram}, and one is a copy of vertex $1$. 
 Levels $V_2$ and $V_3$ each contain four vertices, with two copies of each vertex from diagram $B$.}]{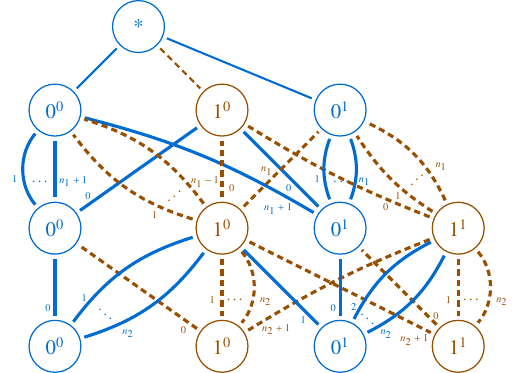}
 \caption{First levels of an ordered Bratteli diagram $\B$ obtained by copy-pasting diagram $B$ from Figure~\ref{fig: general sturmian bv diagram}.
 On each level, vertices sharing the same color---that is, those whose labels have the same base (e.g., $0^0$ and $0^1$)---are copies of the correspondingly coloured vertex (also identified by the base label, disregarding the exponent) on the same level in $B$.}
 \label{fig: extension of sturmian}
\end{center}
\end{figure}
\subsection{Extension triples}\label{sec: collapsing map}
Given diagrams $B$ and $\B$ as above, we call
\begin{align*}
 \pi\: X_{\B}\to X_{B},\qquad (e_{\ell_n,m_n}^{j_n}(n))_{n\geq 0}\mapsto
 (e_{\ell_n,m_n}(n))_{n\geq 0}
\end{align*}
the associated \emph{collapsing map} (compare this to the terminology in \cite{GjerdeJohansen2000})
and refer to $(B,\hat B,\pi)$ as an \emph{extension triple}.
For the well-definition of $\pi$, observe that by \eqref{item: copies of egdes} and \eqref{item: copied edges are connected like their originals} from above, we have $s(e_{\ell_{n+1},m_{n+1}}^{j_{n+1}}(n+1))=r(e_{\ell_n,m_n}^{j_n}(n))=v_{\ell_{n}}^{j_{n}}(n+1)$ only if
$s(e_{\ell_{n+1},m_{n+1}}(n+1))=v_{\ell_{n}}(n+1)=r(e_{\ell_n,m_n}(n))$ so that indeed, $\pi$ maps paths to paths.

\begin{rem}\label{ref: rank is bounded by number of copies}
 Consider an extension triple $(B,\B,\pi)$ and let $n\in \N$.
 For later reference, we remark the obvious fact that
 if $\alpha\in X_B$ traverses infinitely many vertices with 
 at most $n$ copies in $\B$, then
 $|\pi^{-1}\alpha|\leq n$.
 In fact, 
 \[|\pi^{-1}\alpha|=\lim_{n\to \infty}|\{v\in \V_n\: \text{there is } \beta\in \pi^{-1}\alpha \text{ with } s(\beta_n)=v\}|.\]
\end{rem}

We call two extension triples $(B,\B,\pi)$ and $(B',\B',\pi')$ \emph{isomorphic} if there are isomorphisms
$h\: (X_{B},\phi_{B})\to (X_{B'},\phi_{B'})$ and $\hat h\: (X_{\hat B},\phi_{\hat B})\to (X_{\B'},\phi_{\B'})$ such that $\pi'= h\circ \pi\circ {\hat h}^{-1}$.\footnote{Using methods from \cite{HermanPutnamSkau1992}, one can show that the assumption on $\pi'$ is superfluous.
As we don't make use of this fact, we don't provide a proof here.\label{footnote: 1}}
Let us mention the obvious but important fact that if we telescope $B$ and $\B$ along one and the same sequence to obtain diagrams $B'$ and $\B'$, respectively, then $\hat B'$ is obtained by copy-pasting $B'$ and---denoting the corresponding collapsing map by $\pi'$---$(B',\B',\pi')$ is isomorphic to $(B,\B,\pi)$.

It is easy to see that $\pi$ is continuous and onto.
In fact, $\pi$ is a factor map.
To see this, consider a non-maximal path $(\gamma_n)_n=(e_{\ell_n,m_n}^{j_n}(n))_n\in X_{\B}$.
Let $n_0$ be the first level where $\gamma_{n_0}$ is not maximal so that $m_{n_0}<r_{\ell_{n_0}}-1$.
Then
\begin{align*}
&\pi\circ \phi_{\B} ((\gamma_n)_{n\geq0})\\
&=\pi(
e_{i_{0},0}^{k_{0}}(0),
\ldots,
e_{i_{n_0-1},0}^{k_{n_0-1}}(n_0-1),
e_{\ell_{n_0},m_{n_0}+1}^{j_{n_0}}(n_0),
e_{\ell_{n_0+1},m_{n_0+1}}^{j_{n_0+1}}(n_0+1),\ldots
)\\
&=
(
e_{i_{0},0}(0),
\ldots,
e_{i_{n_0-1},0}(n_0-1),
e_{\ell_{n_0},m_{n_0}+1}(n_0),
e_{\ell_{n_0+1},m_{n_0+1}}(n_0+1),\ldots
)\\
&=\phi_B \circ \pi (\gamma_n)_{n\geq 0},
\end{align*}
for appropriate $k_0,\ldots,k_{n_0-1}$ and $i_0,\ldots,i_{n_0-1}$.
As the set of non-maximal paths is dense in $X_{\B}$, this implies
$\pi\circ \phi_{\B}=\phi_B \circ \pi$ by continuity.
Finally, in addition to being a factor map, observe that $\pi$ only maps the minimal (maximal) path in $X_{\B}$ to the minimal (maximal) path in $X_B$---recall that $B$ and $\B$ are throughout assumed to be properly ordered.

Altogether, the above discussion shows that if $\B$ is obtained by copy-pasting $B$, then the collapsing map $\pi\: X_{\B}\to X_B$ is an almost $1$-to-$1$ factor map.
As a matter of fact, we have the following converse, whose proof
is a basic application of the methods developed in \cite{HermanPutnamSkau1992}
and which, at least implicitly, is contained in the proof of \cite[Theorem~3.1]{Sugisaki2011}.\footnote{Note that in the special case of Toeplitz flows (symbolic almost $1$-to-$1$ extensions of odometers), Theorem~\ref{thm Sugisaki} reduces to \cite[Theorem~8]{GjerdeJohansen2000}.
What we call \emph{copy-pasting} here simplifies to the \emph{equal path number property} in \cite{GjerdeJohansen2000}.}
\begin{thm}[{\cite[Theorem~3.1]{Sugisaki2011}}]\label{thm Sugisaki}
 Suppose $(X,f)$ and $(\hat X,\hat f)$ are Cantor minimal systems
 and $q\:(\hat X,\hat f)\to (X,f)$ is an almost $1$-to-$1$ factor map.
 Let $(B,h)$ be a Bratteli-Vershik representation of $(X,f)$
such that $q^{-1}h(\gamma^-)$ is a regular
 $q$-fibre, where $\gamma^-$ is the minimal path in $B$.
 
 Then there is a representation $(\B,\hat h)$ of $(\hat X,\hat f)$ such that $(B,\B,h^{-1}\circ q\circ {\hat h})$
 is an extension triple.
\end{thm}

\begin{rem}
As mentioned in footnote \ref{footnote: 1}, one can show that two
extension triples $(B,\B,\pi)$ and $(B',\B',\pi')$ are necessarily isomorphic whenever
$(X_{B},\phi_{B})$ and $(X_{B'},\phi_{B'})$ as well as $(X_{\hat B},\phi_{\hat B})$ and $(X_{\B'},\phi_{\B'})$ are isomorphic.
Moreover, it is easy to see that if $(B,\B,\pi)$ and $(B',\B',\pi')$ are isomorphic, then $\pi$ is irregular if and only if $\pi'$ is.
With Theorem~\ref{thm Sugisaki}, this has an interesting consequence: 
as long as we deal with almost $1$-to-$1$ Cantor extensions of Cantor minimal systems, (ir)regularity is independent of the specific factor map.
The author is not aware of this fact having been observed elsewhere.
\end{rem}

\begin{rem}\label{rem: adapted to a partition}
Suppose we are in the situation of Theorem~\ref{thm Sugisaki}, and note that the collection $C_0$ of all $0$-cylinders in $X_{\hat B}$ partitions $\hat X=\hat h(\hat X_{\B})$.
Below (specifically, in the second part of the proof of Theorem~\ref{thm: translation to almost automorphic shifts}), given some prescribed clopen partition $P=\{P_1,P_2,\ldots,P_\ell\}$ of $\hat X$, we may want that $(\B,h)$ is \emph{adapted} to $P$, that is, we may want that $C_0$ (or rather, its image under $\hat h$) coincides with $P$.

Clearly, 
we always find a representation where $C_0$
is finer than $P$---we just have to telescope $\B$ between level $0$ and some sufficiently high level $n$.
That is, we may assume without loss of generality that $\B$ is such that for each $e\in \hat E_0$, there is $i$ with $[e]\ssq P_i$ (by which we actually mean, $\hat h[e]\ssq P_i$).
Now, by suitably introducing one new level between level $0$ and level $1$ (with one vertex for each element of $P$), we readily obtain a representation $(\hat B',{\hat h}')$ of $(\hat X,\hat f)$ which is adapted to $P$.

Manipulating $B$ simultaneously in a similar way, we see that it is always possible to obtain representations $(B',h')$ and $(\B',{\hat h}')$ of $(X,f)$ and $(\hat X,\hat f)$, respectively, such that $\B'$ is adapted to $P$, $B'$ is adapted to some prescribed clopen partition $Q$ of $X$ for which $q^{-1}(Q)$ is coarser than (or equal to) $P$, and
$(B',\hat B',{h'}^{-1}\circ q\circ {\hat h}')$ is an extension triple which is isomorphic to the original triple $(B,\hat B,{h}^{-1}\circ q\circ {\hat h})$.
\end{rem}

Given an extension triple $(B,\B,\pi)$, we define
\begin{align}\label{eq: defn discontinuities}
D=\{\alpha\in X_B\: \exists\, \beta=(\beta_n)_{n\geq 0},\beta'=(\beta_n')_{n\geq 0}\in\pi^{-1}\alpha \text{ with } \beta_0\neq \beta'_0\}.
\end{align}
It is easy to see that $D$ is closed with empty interior and thus, in particular, nowhere dense.

Note that $\pi$ is irregular if $\mu(D)>0$ for all invariant measures $\mu$ of $(X_{B},\phi_{B})$.
This is not a necessary criterion: recall that if $(B,\B,\pi)$ and $(B',\B',\pi')$ are isomorphic, then $\pi$ is irregular if and only if $\pi'$ is.
However, given any extension triple
$(B,\B,\pi)$, by adding a new lowest level $V_{-1}$ ($\hat V_{-1}$) below level $0$ to $B$ ($\B$) with only one vertex which is connected to the sole vertex of $V_0$ ($\hat V_0$) by exactly one edge, we obtain an extension triple $(B',\B',\pi')$ which is isomorphic to $(B,\B,\pi)$ while the corresponding set $D$ is empty.

Nonetheless, if $(B,\B,\pi)$ is an extension triple with $\pi$ irregular and
such that $(X_{\B},\phi_{\B})$ is $0$-expansive (see the paragraph preceding Theorem~\ref{thm: 2 to 1 extensions exist} in Section ~\ref{sec: irregular extension}), then it is easy to see that
$\mu(D)$ is necessarily positive for each invariant measure $\mu$ of $(X_{B},\phi_{B})$.
More generally, if $(B,\B,\pi)$ is an extension triple with $\pi$ irregular and
such that $(X_{\B},\phi_{\B})$ is expansive or $(X_{B},\phi_{B})$ has only finitely many ergodic measures, it is not hard to see that through telescoping, we can always obtain an isomorphic extension triple $(B',\B',\pi')$ where $\mu(D)>0$ for each invariant measure $\mu$ of $(X_{B'},\phi_{B'})$.
If $(X_{\B},\phi_{\B})$ is not expansive and $(X_{B},\phi_{B})$ has infinitely many distinct ergodic measures, we can still obtain an isomorphic extension triple $(B',\B',\pi')$
for each finite collection of invariant measures $\mu_1,\ldots,\mu_n$ of $(X_{B},\phi_{B})$
such that
$\mu_i(D)>0$ $(i=1,\ldots,n)$ (where we identify the invariant measures of $(X_{B},\phi_{B})$ and $(X_{B'},\phi_{B'})$ in the obvious way).

Given $\alpha\in X_B$, set 
\[
S_{D}^n(\alpha)=\frac1n\cdot \sum_{i=0}^{n-1}\mathbf 1_{D}(\phi_B^i \alpha)\  \text{ for } n\in \N \ \text{ and } \ S_D(\alpha)=\bigcap_{N\in \N}\overline{\left\{S_{D}^n(\alpha)\: n\geq N\right\}}. 
\]
That is, $S_D(\alpha)$ is the collection of all asymptotic frequencies of visits which $\alpha$ pays to $D$ under iteration of $\phi_B$.
We denote the union of all such collections by
\begin{align}\label{eq: defn birkhoff spectrum}
 \SD=\bigcup_{\alpha\in X_B}S_D(\alpha)
\end{align}
and refer to $\SD$ as the \emph{Birkhoff spectrum} of $(B,\B,\pi)$.

\subsection{Extended Bratteli diagrams}\label{sec: extended BV}
Given an extension triple $(B,\hat B,\pi)$,
we next define an associated extended Bratteli diagram.
This extends the respective notion from \cite{FuhrmannKellendonkYassawi2021}.

Given an extension triple $(B,\B,\pi)$, let 
$\mc B'=(\mc V,\mc E)$ with $\mc V=\bigcup_{n\geq 0} \mc V_n$ and
$\mc E=\bigcup_{n\geq 0} \mc E_n$
be the infinite graph
whose  vertices $\mc V$, edges $\mc E$, and associated range and source maps $r,s\: \mc V\to \mc E$ are as follows.
\begin{enumerate}
\item 
For $n\geq 0$, $\mc V_n=\bigcup_{v\in V_n}\mc V_n^v$, where $\mc V_n^v$ is the collection of all non-empty subsets of $\V_n$ which only contain copies of the vertex $v\in V_n$.
\item Given $n\geq0$, $A=\{v_{k}^{i_1},\ldots,v_{k}^{i_s}\}\in \mc V_{n}$, and $B=\{v_{\ell}^{j_1},\ldots,v_{\ell}^{j_t}\}\in \mc V_{n+1}$, $\mc E_n$ contains exactly one edge $e$ with source $s(e)=A$ and
range $r(e)=B$ for each $e_{\ell,m}\in E_n$ with $s(\{e_{\ell,m}^{j_1},\ldots,e_{\ell,m}^{j_t}\})=A$ (where $s$ is the source map of $\B$).
We may refer to such $e$ as $e_{\ell,m}^{\{j_1,\ldots,j_t\}}$.
\suspend{enumerate}
While the above satisfies the assumptions (1)--(2) of a Bratteli diagram,
we may well have vertices which are not the source of any edge in $\mc E$; that is, assumption (3) may be violated. 
However, by removing all vertices from $\mc B'$ which
are not traversed by an infinite path in $\mc B'$, we obtain a Bratteli diagram $\mc B$.
For simplicity, we keep denoting the resulting collections of vertices and edges by $\mc V$, $\mc V_n$ and
$\mc E$, $\mc E_n$, respectively.
See Figure~\ref{fig: extended BV} for the extended Bratteli diagram
for $B$ and $\B$ from Figure~\ref{fig: general sturmian bv diagram}
and Figure~\ref{fig: extension of sturmian}, respectively.
 \begin{figure}[ht]
\begin{center}
 \includegraphics[scale=0.97, alt={A graph showing vertices arranged by levels $V_0$ through $V_3$, corresponding to Figures~\ref{fig: general sturmian bv diagram} and \ref{fig: extension of sturmian}. 
 Level $V_0$ contains a single vertex. 
 Level $V_1$ has four vertices: the three vertices from diagram $\hat B$ in Figure~\ref{fig: extension of sturmian}, along with one additional vertex labelled $\{0^0,0^1\}$. 
 Levels $V_2$ and $V_3$ each contain six vertices: the four from $\hat B$, together with two additional vertices labelled $\{0^0,0^1\}$ and $\{1^0,1^1\}$, respectively.}]{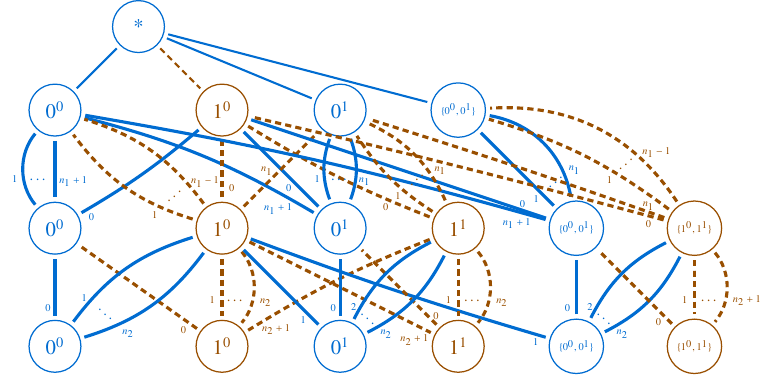}
 \caption{First levels of the extended Bratteli diagram corresponding to the diagrams $B$ and $\B$ from Figure~\ref{fig: general sturmian bv diagram} and Figure~\ref{fig: extension of sturmian}, respectively.
 Observe that while all vertices of the form $\{0^0,0^1\}$ and $\{1^0,1^1\}$ are shown in this figure, whether some of them should be removed---due to not being traversed by any infinite path---is determined by the structure at successive levels.}
 \label{fig: extended BV}
\end{center}
\end{figure}

Note that in contrast to our standard assumptions on
Bratteli diagrams, $\mc B$ is not simple (unless $\pi$ is [strictly] $1$-to-$1$).
Moreover, the space $X_{\mc B}$ of infinite paths
starting in $\mc V_0$ may have isolated points and hence, may not be a Cantor set.
In any case,
\resume{enumerate}
\item $\mc E$ inherits the order from $B$, that is, $e_{\ell,m}^{\{j_1,\ldots,j_t\}}<e_{\ell',m'}^{\{j_1',\ldots,j_t'\}}$ if and only if $\{j_1,\ldots,j_t\}=\{j_1',\ldots,j_t'\}$ and
$e_{\ell,m}<e_{\ell',m'}$.
\end{enumerate}
We call the ordered Bratteli diagram $\mc B$, defined in this manner, the \emph{extended Bratteli diagram} of $(B,\hat B,\pi)$.
Note that $X_{\mc B}$ has a unique minimal (maximal) path so that
the Vershik map $\phi_{\mc B}$ defines a homeomorphism on $X_{\mc B}$.
As $\mc B$ is not simple, $(X_{\mc B},\phi_{\mc B})$ is not minimal.

Identifying singleton sets with their unique element, $\mc B$ can be seen to contain $\B$ as a sub-diagram. 
In this sense, $X_{\mc B}$ contains $X_{\B}$.
Further, the collapsing map $\pi$ naturally extends to a map from $X_{\mc B}$ to $X_B$, which we denote by the same letter $\pi$.
Just as before, one can see that $\pi$ is a factor map.

We write $\mc V_1^+=\mc V_1\setminus \hat V_1$---where, as already mentioned, we identify the singleton vertices in $\mc V_1$ with $\hat V_1$.
That is, $\mc V_1^+$ is the set of vertices in $\mc V_1$ which contain at least two elements.
Given $\alpha\in X_{B}$, note that
 $\alpha\in D$ if and only if there is some $\gamma\in \pi^{-1}(\{\alpha\})$ in $X_{\mc B}$ with $r(\gamma_0)\in \mc V_1^+$.
 We may hence, compute the frequency of visits of $\alpha$ to $D$---or rather, a lower bound for it---by
 computing the frequency of visits of $\gamma$ to $\mc V^+_1$.
 Specifically, given a path $\gamma$ in $X_{\mc B}$, we set
 \begin{align}\label{eq: defn Sn}
  S^n(\gamma)=\frac1n\cdot \sum_{i=0}^{n-1}\mathbf 1_{\mc V_1^+}\phi_{\mc B}^i(\gamma) \quad \text{ for  } n\in \N,
 \end{align}
 where $\mathbf 1_{\mc V_1^+}$ denotes---slightly abusing notation---the indicator function of the set $\{(\gamma_0,\gamma_1,\ldots)\in X_{\mc B}\: r(\gamma_0)\in \mc V_1^+ \}$.
 \begin{prop}\label{prop: Sn is lower bound for SDn}
    Let $(B,\B,\pi)$ be an extension triple and $\mc B$ its extended Bratteli diagram.
    For each $\gamma\in X_{\mc B}$ and all $n\in \N$, $S^n(\gamma)\leq S_D^n(\pi(\gamma))$.
 \end{prop}
\begin{proof}
Note that if $\phi_{\mc B}^i\gamma\in \mc V_1^+$, there are at least two paths $\beta,\beta'\in X_{\B}$ with $\beta_0\neq \beta_0'$ and $\pi(\beta)=\pi(\beta')=\pi(\phi_{\mc B}^i\gamma)=\phi^{i}_{B}\pi(\gamma)$.
 In other words, for each $i\geq 0$ with $\phi_{\mc B}^i\gamma\in \mc V_1^+$, we have 
 $\phi^i_B(\pi(\gamma))\in D$.
\end{proof}

In the other direction, we have
 \begin{prop}\label{prop: full preimage}
  Let $(B,\B,\pi)$ be an extension triple and $\mc B$ its extended Bratteli diagram.
   Consider $\alpha=(\alpha_n)_{n\geq 0}=(e_{\ell_n,m_n})_{n\geq 0}\in X_B$ and
   suppose for each $n\geq0$, $j_1(n),\ldots,j_{N(n)}(n)$ are such that 
   $\{v_{\ell_n}^{j_1(n)},\ldots,v_{\ell_n}^{j_{{N(n)}}(n)}\}$ equals
   \begin{align*}
&\{v\in \V_{n+1}\:\text{there is } \beta\in X_{\B} \text{ with } \pi(\beta)=\alpha \text{ and } r(\beta_{n})=v\}\\
&=\{v\in \V_{n+1}\:\text{there is } \beta\in X_{\B} \text{ with } \pi(\beta)=\alpha \text{ and } s(\beta_{n+1})=v\}.
    \end{align*}
   Then, for each $n\geq0$,
    $\{v_{\ell_n}^{j_1(n)},\ldots,v_{\ell_n}^{j_{{N(n)}}(n)}\}$ is a vertex in $\mc V_{n+1}$.
    Specifically, there is a path $\gamma=(\gamma_n)_{n\geq 0}$ in $X_{\mc B}$ with
   \begin{align}\label{eq: defn maximal preimage}
    \gamma_n=e_{\ell_n,m_n}^{\{j_1(n),\ldots,j_{N(n)}(n)\}} \qquad (n\geq0),
   \end{align}
 and 
 any $\gamma'\in \pi^{-1}(\alpha)\ssq X_{\mc B}$ which coincides with $\gamma$ on infinitely many levels equals $\gamma$.
 
 Moreover, $\gamma$ satisfies $S^n(\gamma)=S^n_D(\alpha)$ for all $n\in \N$.
\end{prop}
\begin{rem}\label{rem: full preimages in finite-to-1 extensions}
 We may call $\gamma$ the \emph{full preimage} of $\alpha$ (to avoid using the term \emph{maximal} in the present context).
 Note that the above statement gives that if the collapsing map $\pi\: X_{\B} \to X_B$ has fibres of cardinality at most $n$, then any path 
$\gamma\in X_{\mc B}$ which traverses  a vertex with $n$ elements is necessarily a full preimage of its projection $\pi(\gamma)\in X_B$.
\end{rem}

\begin{proof}
First, we discuss why $\gamma$, defined via \eqref{eq: defn maximal preimage}, is indeed a path, that is, we discuss $s(\gamma_{n})=r(\gamma_{n-1})$ for $n\in \N$. 
 By definition, for each $n\in \N$ and each $v\in r(\gamma_{n-1})$, there is $\beta\in X_{\B}$ with $\pi(\beta)=\alpha$ and $r(\beta_{n-1})=v$.
 Clearly, $r(\beta_{n-1})=s(\beta_{n})\in s(\gamma_n)$ so that $r(\gamma_{n-1})\subseteq s(\gamma_n)$.
 The other inclusion works similarly, and $\gamma$ is thus a path in $X_{\mc B}$.
 The second part ($\gamma=\gamma'$) readily follows from the simple observation that given $\alpha\in X_B$ and $n\in \N$, the $n$-head of any $\gamma'\in X_{\mc B}$ with $\pi(\gamma')=\alpha$ is uniquely determined by $r(\gamma'_n)$.
 Note that this further implies
 that the iterate $\phi_{\mc B}^i \gamma$ ($i\in\Z$) of the full preimage of a path $\alpha$ is the full preimage of $\phi_B^i\alpha$.
 
 For the ''moreover``-part, 
consider $i\geq 0$ such that $\phi^i_B\alpha\in D$; if no such $i$ exists, $S^n_D(\alpha)=0$ for all $n$ and the statement holds due to Proposition~\ref{prop: Sn is lower bound for SDn}.
 Take $\beta,\beta'\in X_{\B}$ with $\pi(\beta)=\pi(\beta')=\phi_B^i\alpha$ and $\beta_0\neq \beta_0'$ (so that $r(\beta_0)\neq r(\beta_0')$).
 Note that $r((\phi_{\mc B}^i \gamma)_0)\supseteq \{r(\beta_0),r(\beta_0')\}$.
 In particular, $\phi_{\mc B}^i(\gamma)\in \mc V_1^+$.
 As $i$ was arbitrary, $S^n(\gamma)\geq S^n_D(\alpha)$ for all $n\geq 0$. 
 With Proposition~\ref{prop: Sn is lower bound for SDn}, we get $S^n(\gamma)= S^n_D(\alpha)$. 
\end{proof}

\subsection{Irregular extensions}\label{sec: irregular extension}
As an application of the discussion so far, we next show that every Cantor minimal system allows for an  irregular almost \aaa\ extension.
This statement is interesting in its own right (see Remark~\ref{rem: entropy conjecture model sets}) but moreover, shows that the assumptions of Theorem~\ref{thm: main} below are meaningfully satisfied.
Further, its proof---while technically less demanding---can be seen as a precursor to the construction in Section~\ref{sec: non-maximal entire section}.

For $v\in V_n$ with $n\geq 1$, set 
\[
 E^D(v_0,v)=\{(e_0,\ldots,e_{n-1})\in E(v_0,v)\: [e_0,\ldots,e_{n-1}]\cap D\neq \emptyset\}
\]
and 
\[
 X_v^D=\{\alpha\in X_v\: (\alpha_0,\ldots,\alpha_{n-1})\in E^D(v_0,v)\}.
\]
Clearly, for each $(e_0,\ldots,e_{n-1})\in E^D(v_0,v)$ there 
is $\gamma\in X_{\mc B}$ such that $\pi(\gamma)\in [e_0,\ldots,e_{n-1}]$ and $r(\gamma_0)\in \mc V_1^+$.
\begin{prop}
Consider an extension triple $(B,\B,\pi)$ and $D$ as in \eqref{eq: defn discontinuities}.
Then
\begin{align}\label{eq: D}
 D=\bigcap_{n\in \N}\bigcup_{v\in V_n}X_v^D.
\end{align}
\end{prop}
\begin{proof}
 The inclusion $\ssq$ is immediate.
 For the other inclusion, consider a path $\alpha=(\alpha_0,\alpha_1,\ldots)\in \bigcap_{n\in \N}\bigcup_{v\in V_n}X_v^D$.
Note that for every $n\in \N$, we have some $\gamma^{(n)}\in X_{\mc B}$ with $r{(\gamma^{(n)}_0)}\in \mc V_1^+$
and $\pi(\gamma)\in [\alpha_0,\alpha_1,\ldots,\alpha_{n-1}]$.
Without loss of generality, we may assume that $\gamma^{(n)}$ converges to some $\gamma\in X_{\mc B}$.
Then $\pi(\gamma)=\alpha$ and $r(\gamma_0)\in \mc V^+_1$, that is, $\alpha\in D$.
\end{proof}
As an immediate consequence, we get
\begin{cor}\label{cor: irregular extension}
 Consider an extension triple $(B,\B,\pi)$ and $D$ as in \eqref{eq: defn discontinuities}.
 Suppose $\mu$ is an invariant measure of $(X_B,\phi_B)$.
 Then
 \[
  \mu(D)=\lim_{n\to \infty} \sum_{v\in V_n}\frac{|E^D(v_0,v)|}{|E(v_0,v)|} \cdot\mu(X_v).
 \]
 In particular, 
 \[
  \mu(D)\geq \varlimsup_{n\to \infty} \min_{v\in V_n}\frac{|E^D(v_0,v)|}{|E(v_0,v)|}.
 \]

\end{cor}
\begin{proof}
 With the previous statement, we have
 \begin{align*}
\mu(D)&=\mu(\bigcap_{n\in \N}\bigcup_{v\in V_n}X_v^D)=\lim_{N\to \infty} \mu(\bigcap_{n=1}^N\bigcup_{v\in V_n}X_v^D)
=\lim_{N\to \infty} \mu(\bigcup_{v\in V_N}X_v^D)
\\
&=\lim_{N\to \infty} \sum_{v\in V_N}\mu(X_v^D) =\lim_{n\to \infty} \sum_{v\in V_n}\frac{|E^D(v_0,v)|}{|E(v_0,v)|} \cdot\mu(X_v),
 \end{align*}
 where we used \eqref{eq: inv measure} in the last step.
The statement follows.
\end{proof}

Recall that a dynamical system $(X,f)$ is \emph{expansive} if for some (and hence, any) compatible metric $d$, there is $\delta>0$ such that for all $x\neq y\in X$ there is $n\in \Z$ with $d(f^nx,f^ny)>\delta$.
For a Bratteli-Vershik system $(X_B,\phi_B)$, expansivity
is equivalent to the existence of some $k\geq 0$ such that for each pair of paths $\gamma\neq \gamma'\in X_B$, there is $n\in \Z$ such that the $k$-heads of $\phi_B^n(\gamma)$ and $\phi_B^n(\gamma')$ disagree.
In this case, we also say that $(X_B,\phi_B)$ is \emph{$k$-expansive}.

\begin{thm}\label{thm: 2 to 1 extensions exist}
For each simple properly ordered Bratteli diagram $B$, telescoping yields a diagram $B'$ admitting an extension triple $(B',\B,\pi)$ where $\pi$ is
an irregular almost \aaa\ factor map.
If, additionally, $(X_B,\phi_B)$ is $k$-expansive, then $\B$ can be chosen such that also $(X_{\B},\phi_{\B})$ is $k$-expansive.
%
\end{thm}
\begin{rem}\label{rem: entropy conjecture model sets}
 Before proving the above statement, let us briefly discuss some important consequence related to the entropy conjecture for irregular cut and project schemes due to Moody, see \cite{BaakeJagerLenz2016,JaegerLenzOertel2016PositiveEntropyModelSets,FuhrmannGlasnerJagerOertel2021} and references therein.
It is a classical fact that if $X$ is a Cantor set, then $(X,f)$ is expansive if and only if it is isomorphic to a subshift $(\Sigma,\sigma)$, where
$\Sigma\ssq \mc A^\Z$ is a closed (in the product topology) collection of bi-infinite sequences
over the finite alphabet $\mc A$ and $\s$ denotes the left-shift \cite{Hedlund1969}.
 Therefore, Theorem~\ref{thm: 2 to 1 extensions exist} gives that to each minimal subshift $(\Sigma,\sigma)$ we can construct a minimal subshift $(\hat\Sigma,\sigma)$ which is an irregular almost $1$-to-$1$ extension with fibres of cardinality not bigger than $2$ and, consequently,
 with topological entropy equal to that of $(\Sigma,\sigma)$, see \cite{bowen1971}.
 
If, additionally, $(\Sigma,\sigma)$ is 
an almost $1$-to-$1$ extension of a minimal rotation $(\T,\g)$,
then $(\hat \Sigma,\sigma)$ is an \emph{irregular} almost $1$-to-$1$ extension of $(\T,\g)$, see also Section~\ref{sec: aa subshifts}.

Now, each minimal rotation $(\T,\g)$ allows for a subshift which is a regular almost $1$-to-$1$ extension \cite[Corollary~3.13]{FuhrmannKwietniak2020} and hence, of entropy $0$ (for the variational principle).
Extending such a subshift
as in Theorem~\ref{thm: 2 to 1 extensions exist}, we see that each minimal rotation $(\T,\g)$ admits an
\emph{irregular} almost $1$-to-$1$ extension with vanishing entropy.
Altogether, this gives an alternative to the constructions in \cite{BaakeJagerLenz2016,FuhrmannGlasnerJagerOertel2021}
and, in fact, an extension of the respective results:
Essentially, it shows that 
the existence of an irregular model set of vanishing entropy does not impose a restriction on the internal space of a cut and project scheme with external space $\R$ or $\Z$.

As the technical details of this discussion are beyond the scope of the present article, we refer the interested reader to
\cite{Robinson2007,BaakeLenzMoody2007Characterization} for a background on model sets and the cut and project scheme.
\end{rem}

\begin{proof}[Proof of Theorem~\ref{thm: 2 to 1 extensions exist}]
Consider an ordered diagram $B=(V,E,\leq)$ as in the assumptions.
For $n\in \N$ and $v_\ell\in V_n$ (with $\ell\in\{0,\ldots,|V_n|-1\}$), set 
\begin{align*}
E^{\textrm{ex}}(v_0,v_\ell)\!=\{(e_0,\ldots,e_{n-1})\in E(v_0,v_\ell)\: e_i \text{ is extremal for some } 0<i<n\},
\end{align*}
where we call  an edge \emph{extremal} if it is minimal or maximal.
Obviously, $E^{\textrm{ex}}(v_0,v_\ell)=\emptyset$ if $v_\ell\in V_1$.
Further, note that for $n\geq 2$
\begin{align*}
 \frac{|E^{\textrm{ex}}(v_0,v_\ell)|}{|E(v_0,v_\ell)|}\!=\!\frac{|E(v_0,s(e_{\ell,0}))|+|E(v_0,s(e_{\ell,r_\ell-1}))|}{|E(v_0,v_\ell)|}+\frac{\sum_{i=1}^{r_\ell-2}|E^{\textrm{ex}}(v_0,s(e_{\ell,i}))|}{\sum_{i=0}^{r_\ell-1}|E(v_0,s(e_{\ell,i}))|}.
\end{align*}
Through telescoping and, if necessary, relabelling, we may assume without loss of generality that $B$ satisfies
\begin{enumerate}[(a)]
\item For each $n\geq 2$ and $v\in V_n$, $|{E^{\textrm{ex}}(v_0,v)}|/|{E(v_0,v)}|<1/2$ (this is always possible, as shown by a simple induction using the above equality);
\item Vertices from consecutive levels are connected by at least $3$ edges;
\item For each $n\geq 1$, $v_0(n)\in V_n$ is the unique source of the minimal edges.
\end{enumerate}

 Let $\hat B=(\hat V,\hat E,\hat\leq)$ be the properly ordered Bratteli diagram obtained by copy-pasting $B$, satisfying the following properties.
\begin{enumerate}[(i)]
 \item For each $n\geq1$, $\V_n$ contains two copies
 $v_\ell^0(n),v_\ell^1(n)$ of each vertex $v_\ell(n)\in V_n$.
 Only for $\ell=0$, there is an additional third copy $v_0^2(n)$ of $v_0(n)$ (which serves as the unique source of the minimal edges in $\E_n$---see the next item). 
 \item For each $n\geq 1$ and each $j$ and $\ell$, $s(\e^j_{\ell,0}(n))=v_0^2(n)$ and $s(\e^j_{\ell,r_\ell-1}(n))=v_{\ell_+}^0(n)$,
 where $\ell_+=\ell_+(n,\ell)$ is such that $v_{\ell_+}(n)=s(e_{\ell,r_\ell-1}(n))$ in $B$.
 \item 
 For each $n\geq 1$ and $m=1,\ldots,r_0(n)-2$, $s(\e^2_{0,m}(n))=v_{\ell}^j(n)$
 assuming that $s(\e_{0,m}(n))=v_\ell(n)$ (in $B$),
 where $j\in \{0,1\}$ is such that
 \[j=|\{k=0,\ldots,m-1\:s(e_{0,k}(n))=v_\ell(n)\}|\bmod 2.\]
 Note that due to (b) and the previous item, this ensures that $\B$ is simple.
 \item
 For each $n\geq 1$, $j\in\{0,1\}$, each $\ell$, and $m=1,\ldots,r_\ell(n)-2$,
 we have $s(\e^j_{\ell,m}(n))=v_{\ell'}^j(n)$ assuming that $s(\e_{\ell,m}(n))=v_{\ell'}(n)$ (in $B$).
  \setcounter{counter}{\value{enumi}}
\end{enumerate}
Let $(B,\B,\pi)$ be the associated extension triple with $D$ as in \eqref{eq: defn discontinuities}.
Note that the vertex $v_0^2(n)$ ($n\in \N$) is the source
of minimal edges only.
Accordingly, if $\beta\in X_{\B}$ traverses $v_0^2(n)$ on infinitely many levels $n$, then $\beta$ is the sole preimage
of $\pi(\beta)$.
With Remark~\ref{ref: rank is bounded by number of copies} and due to item (i), we hence obtain that fibres have no more than $2$ elements, that is, $|\pi^{-1}\alpha|\leq 2$ for all $\alpha\in X_B$.

Further, for $v\in V_n$ ($n\geq 2$),
$E^D(v_0,v)= E(v_0,v)\setminus E^{\textrm{ex}}(v_0,v)$ so that by item (a),
$|E^D(v_0,v)|/|E(v_0,v)|>1/2$.
Corollary~\ref{cor: irregular extension} hence implies $\mu(D)>0$ for every invariant measure $\mu$.
The first part of the statement follows.

%
Now, suppose $(X_B,\phi_B)$ is $k$-expansive (which is preserved under telescoping).
Then, given $\beta,\beta'\in X_{\B}$ with $\alpha=\pi(\beta)\neq \pi(\beta')=\alpha'$, there is
$i\in \Z$ with $(\phi_B^i(\alpha))_k\neq (\phi_B^i(\alpha'))_k$ and hence, $(\phi^i_{\B}(\beta))_k\neq (\phi^i_{\B}(\beta))_k$.
To show that $(X_{\B},\phi_{\B})$ is $k$-expansive, it thus suffices to prove that
given distinct $\beta,\beta'\in X_{\B}$ with $\pi(\beta)=\pi(\beta')=\alpha$, there is $i\in \Z$ with $(\phi_{\B}^i(\beta))_0\neq (\phi_{\B}^i(\beta'))_0$.
 To that end, pick $n_0$ such that 
$\beta_{n_0}\neq \beta_{n_0}'$ and hence, $\beta_{n}\neq \beta_{n}'$ for $n\geq n_0$---note that if $\pi(\beta)=\pi(\beta')$ and $\beta_{n}= \beta_{n}'$ for some $n$, then $\beta_{m}= \beta_{m}'$ for all $m\leq n$.
By our assumptions on $B$, for each $v\in V_n$ ($n\geq1$), there is a path in $E(v_0,v)\setminus E^{\textrm{ex}}(v_0,v)$.
Let $\gamma\in E(v_0,r(\alpha_{n_0}))$ be such a path.
Then there is $i\in \Z$ such that the $n_0$-head of $\phi_B^i(\alpha)$ coincides with $\gamma$ while $\alpha_n=(\phi_B^i(\alpha))_n$ for $n>n_0$.
Consequently, $\phi_{B}^i(\alpha)\in D$ and $(\phi_{\B}^i(\beta))_n\neq(\phi_{\B}^i(\beta'))_n$ for $n\geq 0$.
\end{proof}
\begin{rem}
A straightforward adaptation of the above proof demonstrates that, for each $n\in\N$, we can find a Cantor minimal system that is an (expansive) irregular \aac\ extension of the (expansive) system $(X,f)$.
\end{rem}

\section{Minimal size of fibres implies maximal Birkhoff spectrum}\label{sec: proof of the main theorem}
We now prove our first main result, Theorem~\ref{thm: main}.
As an intermediate step, we derive a finite-time analogue of the statement, Lemma~\ref{lem: finitary argument for interval}.

In general, it is non-trivial to identify the full preimage of a given path in $X_B$.
Therefore, the average $S^n$ from \eqref{eq: defn Sn} will often only allow us to obtain a lower bound on the frequencies $S_D^n$ that we are actually after.
However, the advantage of dealing with $S^n$ over dealing with $S_D^n$ is that
we can make sense of $S^n$ not only for infinite but also for finite paths.
Before making this precise, we introduce some terminology.

 In the following, given finitely many finite paths $\gamma^1,\gamma^2,\ldots,\gamma^n$ with $r(\gamma^i)=s(\gamma^{i+1})$ for $i=1,\ldots,n-1$,
 we denote their concatenation by $(\gamma^1,\gamma^2,\ldots,\gamma^n)$.
\begin{defn}\label{defn: minimal time to maximality}
 Consider an ordered Bratteli diagram $B=(V,E,\leq)$.
 Given $N>0$ and a path $\gamma \in E_{0,N}$, $T(\gamma)$
 denotes the maximal number of times we can iterate $\gamma$ unambiguously, that is, $T(\gamma)\in \Z_{\geq 0}$ is  such that $\phi^{T(\gamma)}(\gamma)$ is maximal.
 Further, given $\gamma \in E_{n,N}$  with $0<n< N$, we set
 $T(\gamma)=\min_{\gamma'} T(\gamma',\gamma)$, where the minimum is taken over all $\gamma'\in E(v_0,s(\gamma))$.
\end{defn}

\begin{rem}
 Obviously, for any $\gamma\in E_{n,N}$ with $0<n<N$, we have $T(\gamma)=T(\gamma',\gamma)$ where $\gamma'$ is the unique maximal path in $E(v_0,s(\gamma))$.
\end{rem} 
For a \emph{finite} path $\gamma$ in $\mc B$ starting at level $0$, we set 
\[S^n(\gamma)=1/n\cdot\sum_{i=0}^{n-1}\mathbf 1_{\mc V_1^+}\phi^i(\gamma)
 \quad \text{for all } n\in\{0,1,\ldots,T(\gamma)\}.
\]

In the following, we call a path $(e_0,e_1,\ldots)$ \emph{pre-maximal} if for each of its edges $e_i$, there is exactly one edge $f_i\in E$ with $f_i>e_i$.
Clearly, every path bigger than a pre-maximal path necessarily contains a maximal edge.

\begin{defn}
 Given an ordered Bratteli diagram $B=(V,E,\leq)$, consider $N>n>0$ and $\delta>0$.
 We say that \emph{$N$ exceeds $n$ on a scale $\delta$} if for every pre-maximal path $\Gamma\in E_{n,N}$ and every path $\gamma\in E_{0,n}$
 \begin{align*}
  T(\Gamma)>T(\gamma)/\delta.
 \end{align*}
\end{defn}
\begin{rem}
 It is obvious but important to note that for each level $n$ and each $\delta>0$,
there is some $N>n$ which exceeds $n$ on a scale  $\delta$.
\end{rem}

\begin{lem}\label{lem: consequence of exceeding levels}
Let $\mc B$ be the extended Bratteli diagram of an extension triple $(B,\B,\pi)$, where the maximal edges on each level in $\B$ have a unique source.
Consider $\gamma\in \mc E_{n,N}$, where $N>n>0$ and $N$ exceeds $n$ on a scale $\delta$.

Then, given any $\gamma',\gamma''\in \mc E(v_0,s(\gamma))$, we have
 \begin{align*}
  |S^{T(\gamma',\gamma)}(\gamma',\gamma)
  -S^{T(\gamma'',\gamma)}(\gamma'',\gamma)|\leq 2\delta.
 \end{align*} 
\end{lem}
\begin{proof}
First of all, note that we may assume without loss of generality that $\gamma$ is not bigger than the pre-maximal path ending in $r(\gamma)$.
For otherwise, $\phi^i(\gamma',\gamma)$ and $\phi^i(\gamma'',\gamma)$ would contain maximal edges for each $i=0,\ldots,T(\gamma',\gamma)$ and $i=0,\ldots,T(\gamma'',\gamma)$, respectively, so that $S^{T(\gamma',\gamma)}(\gamma',\gamma)=S^{T(\gamma'',\gamma)}(\gamma'',\gamma)=0$; here, we use the assumption that maximal edges in $\B$ start in the same vertex.

Now, note that $\phi^{i+T(\gamma')}(\gamma',\gamma)=\phi^{i+T(\gamma'')}(\gamma'',\gamma)$ for $i\geq 0$.
Moreover, $T(\gamma',\gamma)=T(\gamma')+T(\gamma)$ and
$T(\gamma'',\gamma)=T(\gamma'')+T(\gamma)$.
Without loss of generality, we may assume in the following that $T(\gamma')> T(\gamma'')$.
Then,
 \begin{align*}
  &|S^{T(\gamma',\gamma)}(\gamma',\gamma)
  -S^{T(\gamma'',\gamma)}(\gamma'',\gamma)|\\
  &=
  \Big| 1/T(\gamma',\gamma)\cdot \!\!\sum_{i=0}^{T(\gamma',\gamma)-1}\mathbf 1_{\mc V_1^+}(\phi^i(\gamma',\gamma))-
  1/T(\gamma'',\gamma)\cdot\sum_{i=0}^{T(\gamma'',\gamma)-1}\mathbf 1_{\mc V_1^+}(\phi^i(\gamma'',\gamma))\Big|\\
  &=
  \Big| 1/T(\gamma',\gamma)\cdot \Big (\sum_{i=0}^{T(\gamma')-1}\mathbf 1_{\mc V_1^+}(\phi^i(\gamma',\gamma))+\sum_{i=0}^{T(\gamma)-1}\mathbf 1_{\mc V_1^+}(\phi^{i+T(\gamma')}(\gamma',\gamma))\Big)\\
  &\phantom{=}-
  1/T(\gamma'',\gamma)\cdot\Big (\sum_{i=0}^{T(\gamma'')-1}\mathbf 1_{\mc V_1^+}(\phi^i(\gamma'',\gamma))+\sum_{i=0}^{T(\gamma)-1}\mathbf 1_{\mc V_1^+}(\phi^{i+T(\gamma'')}(\gamma'',\gamma))\Big)\Big|\\
&\leq T(\gamma')/T(\gamma)+ 
\Big |  (1/T(\gamma',\gamma)-1/T(\gamma'',\gamma))\cdot 
\sum_{i=0}^{T(\gamma)-1}\mathbf 1_{\mc V_1^+}(\phi^{i+T(\gamma')}(\gamma',\gamma))
\Big|.
 \end{align*}
As $N$ exceeds $n$ on a scale $\delta$ and $\gamma$ is not bigger than the pre-maximal path, $T(\gamma')/T(\gamma)<\delta$
and $T(\gamma',\gamma)-T(\gamma'',\gamma)\leq \delta\cdot T(\gamma)$.
We conclude that
\[
|S^{T(\gamma',\gamma)}(\gamma',\gamma)
  -S^{T(\gamma'',\gamma)}(\gamma'',\gamma)|\leq \delta +
  \delta \cdot T(\gamma)/T(\gamma'',\gamma)\leq 2\delta.
  \qedhere
\]
\end{proof}

\begin{defn}
 A vertex $v$ in the extended Bratteli diagram
 $\mc B$ of an extension triple $(B, \B,\pi)$ can \emph{realise} a frequency $\nu\geq0$
 if there is a finite path $\gamma\in \mc E(v_0,v)$
 such that $S^{T(\gamma)}(\gamma)=\nu$.
Given $\nu_0\geq 0$ and $n\in \N$, let $\mc V_n^+(\nu_0)$ be the collection of vertices in $\mc V_n$ that can realise some $\nu\geq \nu_0$ and let $\mc V_n^-(\nu_0)=\mc V_n\setminus \mc V_n^+(\nu_0)$.
\end{defn}

The following statement is a finitary precursor to Theorem~\ref{thm: main} and an important step towards its proof.

\begin{lem}\label{lem: finitary argument for interval}
Let $\mc B$ be the extended Bratteli diagram of an extension triple $(B,\B,\pi)$, where the maximal edges on each level in $\B$ have a unique source.
 Suppose there is some $\alpha\in X_{B}$ with
 $\lim_{n\to\infty} S_D^n(\alpha)=\nu_0>0$ and consider $\nu\in(0,\nu_0)$. 
 Let $(\delta_i)_{i\geq 2}$ be a null sequence in $(0,\nu/2)$ and suppose that $1<n_1<n_2<\ldots <n_k$ are levels such that $n_{i+1}$ exceeds $n_i$ on a scale $\delta_{i+1}$ for all $i=1,\ldots,k-1$.
 
 Then, $\mc V_{n_k}^+(\nu)\neq \emptyset$ and for any $v\in \mc V_{n_k}^+(\nu)$ there is a finite path $(\gamma^1,\ldots,\gamma^k)$ in $\mc E(v_0,v)$ with $\gamma^1\in \mc E_{0,n_1}$ and $\gamma^{i}\in \mc E_{n_{i-1},n_{i}}$ for $i=2,\ldots,k$
 such that for $i=1,\ldots,k$,
 \begin{align}
  S^{T(\gamma^1,\ldots,\gamma^i)}(\gamma^1,\ldots,\gamma^k)\,
  \begin{cases}\label{eq: finitary version of realising path}
   \geq \nu-2\delta_i & \text{ if } i \text{ is odd},\\
   \leq \nu+2\delta_i & \text{ if } i \text{ is even}.
  \end{cases}
 \end{align}
\end{lem}
\begin{proof}
First, we discuss $\mc V_{n_k}^+(\nu)\neq \emptyset$ or actually, $\mc V_n^+(\nu)\neq \emptyset$ for $n>0$.
To that end, observe that it suffices to show $\mc V_n^+(\nu)\neq \emptyset$ for arbitrarily large $n$.
Now, let $\gamma\in X_{\mc B}$ be such that $\pi(\gamma)=\alpha$ and $S^m(\gamma)=S_D^m(\alpha)$ for all $m\in \N$, see Proposition~\ref{prop: full preimage}.
Then, for sufficiently large levels, the $\phi_{\mc B}$-orbit of $\gamma$ has to pass through a
vertex which can realise a frequency bigger or equal to $\nu$, since otherwise, $S_D^m(\alpha)=S^m(\gamma)<\nu< \nu_0$ for arbitrarily large $m$ in contradiction to the assumptions on $\alpha$.
Hence, $\mc V_n^+(\nu)\neq \emptyset$ for each $n$.
Note also that $\mc V_n^-(\nu)\neq \emptyset$ ($n\in \N$) since singleton vertices can only realise $0$.

Turning to the construction of the paths $\gamma^i$,
we may assume without loss of generality that 
$k$ is odd---the even case can be reduced to the odd one by considering an additional level $n_{k+1}$ which exceeds $n_k$ on a scale $\delta_{k+1}$.
We may further assume that $k\geq3$ (the case $k=1$ is trivial).

We start by constructing $\gamma^k$.
With $v\in \mc V_{n_k}^+(\nu)$ as in the assumptions, pick any path $\tilde \gamma^k\in \mc E(v_0,v)$ with 
$S^{T(\tilde \gamma^k)}(\tilde \gamma^k)\geq \nu$.
Note that there must be some $i\in \{0,\ldots,T(\tilde\gamma^k)-1\}$ such that
$s\big((\phi_{\mc B}^i(\tilde \gamma^k))_{n_{k-1}}\big)\in \mc V_{n_{k-1}}^+(\nu)$ and $s\big((\phi_{\mc B}^i(\tilde \gamma^k))_{n_{k-2}}\big)\in \mc V_{n_{k-2}}^+(\nu)$---otherwise, we had $S^{T(\tilde \gamma^k)}(\tilde \gamma^k)< \nu$.
By possibly iterating forwards by the smallest such $i$, we may assume without loss of generality that $i=0$, that is,
$s(\tilde \gamma^k_{n_{k-1}})\in \mc V_{n_{k-1}}^+(\nu)$ and
$s(\tilde \gamma^k_{n_{k-2}})\in \mc V_{n_{k-2}}^+(\nu)$
and $S^{T(\tilde \gamma^k)}(\tilde \gamma^k)\geq \nu$.
We set $\gamma^k=(\tilde \gamma^k_{n_{k-1}},\ldots,\tilde \gamma^k_{n_{k}-1})$.
Lemma~\ref{lem: consequence of exceeding levels} gives
\begin{align}\label{eq: finitary lemma--approximate nu from above}
S^{T(\gamma',\gamma^k)}(\gamma',\gamma^k)\geq \nu-2\delta_k
\end{align}
for any $\gamma'\in \mc E(v_0,s(\gamma^k))$.

Towards the definition of $\gamma^{k-1}$, recall that every path $\hat \Gamma \in \mc E_{n_{k-2},n_{k-1}}$ with $r(\hat \Gamma)=s(\gamma^k)$ that is bigger than the pre-maximal path 
(with the same range) contains a maximal edge, and hence traverses singleton vertices.
Recall further that $s((\tilde \gamma^k)_{n_{k-2}})\in\mc V_{n_{k-2}}^+(\nu)$ and $r((\tilde \gamma^k)_{n_{k-1}-1})=s(\gamma^k)$ (by definition of $\gamma^k$).
In other words, there are paths in $\mc E_{n_{k-2},n_{k-1}}$ that end in $s(\gamma^k)$ and start in 
$\mc V_{n_{k-2}}^-(\nu)$ (e.g.\ paths like $\hat\Gamma$ from above)
just as there are paths that end in $s(\gamma^k)$ and start in 
$\mc V_{n_{k-2}}^+(\nu)$ (e.g.\ $(\tilde \gamma^k_{n_{k-2}},\ldots,\tilde \gamma^k_{n_{k-1}-1})$).
We define $\gamma^{k-1}$ to be the path that is maximal among the elements of $\mc E_{n_{k-2},n_{k-1}}$ that end in $s(\gamma^k)$ and start in $\mc V_{n_{k-2}}^+(\nu)$.
Then, if $\Gamma\in \mc E(v_0,s(\gamma^{k-1}))$ is maximal, we have 
$S^{T(\Gamma,\gamma^{k-1})}(\Gamma,\gamma^{k-1})<\nu$ since 
$\phi_{\mc B}^i(\Gamma,\gamma^{k-1})$ lies in $\hat V_1$ (viewed as a subset of $\mc V_1$) for $i=0$ and traverses vertices in $\mc V_{n_{k-2}}^-(\nu)$ for $i>0$.
Lemma~\ref{lem: consequence of exceeding levels} gives
\begin{align}\label{eq: finitary lemma--approximate nu from below}
S^{T(\gamma',\gamma^{k-1})}(\gamma',\gamma^{k-1})\leq \nu+2\delta_{k-1}
\end{align}
 for any $\gamma'\in \mc E(v_0,s(\gamma^{k-1}))$.
 
 Now, observe that \eqref{eq: finitary lemma--approximate nu from above} and \eqref{eq: finitary lemma--approximate nu from below} already give \eqref{eq: finitary version of realising path} for $i=k,k-1$ independently of the particular choice for $\gamma^1,\ldots,\gamma^{k-2}$.
 We can hence repeat the above steps to construct $\gamma^{k-2}$ (similarly to how we constructed $\gamma^k$, this time with $v=s(\gamma^{k-1})$) and $\gamma^{k-3}$
 (similarly to how we constructed $\gamma^{k-1}$)
 such that \eqref{eq: finitary version of realising path} is also satisfied for $i=k-3,k-2$. 
 Repeating this procedure finitely many times gives the statement.
\end{proof}

To show Theorem~\ref{thm: main}, we will need the following general, basic fact, whose proof we provide for the convenience of the reader.
\begin{prop}\label{prop: limsup is assumed as limit}
 Let $(X,f)$ be a topological dynamical system and let $\mc M$ be the collection of its invariant measures.
 Suppose $D\ssq X$ is closed.
 
 Then, for any $x\in X$, we have
 \[
\limsup_{n\to\infty} 1/n\cdot \sum_{i=0}^{n-1}\mathbf 1_D(f^i(x))\leq  \sup_{\mu \in \mc M} \mu(D).
 \]
 Moreover, there is $x\in X$ with 
 $\lim_{n\to\infty} 1/n\cdot \sum_{i=0}^{n-1}\mathbf 1_D(f^i(x))=  \sup_{\mu \in \mc M} \mu(D)$.
\end{prop}
\begin{proof}
Suppose for a contradiction that there is $x\in X$ and a strictly increasing sequence $(n_\ell)$ in $\N$ with
$\lim_{\ell\to\infty} 1/n_\ell\cdot \sum_{i=0}^{n_\ell-1}\mathbf 1_D(f^i(x))> \sup_{\mu \in \mc M} \mu(D)$.
By a standard Krylov-Bogolubov argument (and by possibly going over to a subsequence), we may assume without loss of generality that
$\lim_{\ell\to\infty} 1/n_\ell\cdot \sum_{i=0}^{n_\ell-1}\mathbf \delta_{f^i(x)}$ converges in the weak-*topology to some $\nu \in \mc M$.

Now, the Portmanteau Theorem gives 
$\limsup_{\ell\to\infty} 1/n_\ell\cdot \sum_{i=0}^{n_\ell-1} \mathbf 1_D(f^i(x))\leq \nu(D)\leq \sup_{\mu \in \mc M} \mu(D)$ in contradiction to the assumptions on $x$ and $(n_\ell)$.
This proves the first part.

For the ``moreover''-part, let $(\mu_n)$ be a sequence in $\mc M$ such that $\mu_n(D)$ converges to $\sup_{\mu \in \mc M} \mu(D)$.
By weak-*compactness of $\mc M$, we may assume without loss of generality that 
$(\mu_n)$ converges to some $\mu \in \mc M$.
By the Portmanteau Theorem, $\mu(D)\geq \lim_{n\to \infty}\mu_n(D)$ and thus,
$\mu(D)=\sup_{\mu \in \mc M} \mu(D)$.
Now, ${S}_D(x)=\lim_{n\to\infty} 1/n\cdot \sum_{i=0}^{n-1}\mathbf 1_D(f^i(x))$ exists $\mu$-a.s.\ and
$\int\! {S}_D\, d\mu=\mu(D)$ by Birkhoff's Ergodic Theorem.
Hence, there must be $x\in X$ with
$\lim_{n\to\infty} 1/n\cdot \sum_{i=0}^{n-1}\mathbf 1_D(f^i(x))= {S}_D(x)\geq\mu(D)=\sup_{\mu \in \mc M} \mu(D)$ which, together with the first part, proves the statement.
 \end{proof}

 Recall the definition of the Birkhoff spectrum $\SD$ associated to an extension triple $(B,\B,\pi)$ in \eqref{eq: defn birkhoff spectrum}.
 Proposition~\ref{prop: limsup is assumed as limit} implies $\SD\ssq [0,\sup_{\mu\in \mc M}\mu(D)]$
 but also gives $\alpha\in X_B$ with
 $\lim_{n\to\infty} S^n_D(\alpha)= \sup_{\mu\in \mc M}\mu(D)$.
Moreover, also $0\in \SD$ since $D$ is nowhere dense: $\bigcup_{i\in \Z} \phi_B^i(D)$ is meagre and paths in the non-empty complement of $\bigcup_{i\in \Z} \phi_B^i(D)$ never visit $D$.
 Hence, we always have $\{0,\sup_{\mu\in \mc M}\mu(D)\}\ssq \SD\ssq [0,\sup_{\mu\in \mc M}\mu(D)]$.
 
The next statement, together with Theorem~\ref{thm: 2 to 1 extensions exist}, shows that every ordered Bratteli diagram
$B$ admits (up to telescoping) an extension triple $(B,\B,\pi)$ whose Birkhoff spectrum is non-trivially \emph{maximal}: $\{0\}\subsetneq \SD=[0,\sup_{\mu\in \mc M}\mu(D)]$.
\begin{thm}\label{thm: main}
   Let $(B,\B,\pi)$ be an extension triple where $\pi$ is at most $2$-to-$1$.
   Then, $\SD=[0,\sup_{\mu\in \mc M}\mu(D)]$ where $\mc M$ is the collection of
all invariant measures of $(X_B,\phi_B)$.
  \end{thm}
\begin{proof}
First of all, note that if we telescope $B$ and $\B$ simultaneously along a sequence $n_0=0<n_1<n_2<\ldots$ with $n_1=1$, then the Birkhoff spectrum of the telescoped extension triple coincides with that of $(B,\B,\pi)$.
We may hence assume without loss of generality that maximal edges in $\B$ have a unique source.
We may further assume that $\sup_{\mu\in \mc M}\mu(D)>0$ (the other case is trivial).

Second, note that as a consequence of the discussion before the statement, it suffices to show
that $\SD$ contains $(0,\sup_{\mu\in \mc M}\mu(D))$.
To that end, we show that
for each $\nu\in(0,\sup_{\mu\in \mc M}\mu(D))$, there is
$\beta\in X_B$ which satisfies 
\begin{align}\label{eq: z realises all frequencies}
\limsup_{n\to \infty} S_{D}^n(\beta)\geq \nu\qquad \text{ and } \qquad
\liminf_{n\to \infty} S_{D}^n(\beta)\leq \nu.
\end{align}
Note that this proves the statement.
Indeed, given $\eps>0$ and $N\in \N$, pick integers $n_+>n_->N$ such that $1/n_-<\eps$ and $S_{D}^{n_-}(\beta)\leq \nu+\eps$ as well as $S_{D}^{n_+}(\beta)\geq \nu-\eps$.
Then, as $|S_{D}^{n}(\beta)-S_{D}^{n+1}(\beta))|\leq1/n_-<\eps$ for all $n\geq n_-$, we have that
$\{S_{D}^{n_-}(\beta),\ldots ,S_{D}^{n_+}(\beta)\}$ is $\eps$-dense in the interval spanned by its extremal points and has
thus a non-empty intersection with $[\nu-\eps,\nu+\eps]$.
The theorem follows since $\eps>0$ and $N\in \N$ were arbitrary.

Now, given $\nu\in(0,\sup_{\mu\in \mc M}\mu(D))$, pick some null sequence $(\delta_k)_{k\geq 2}$ in $(0,\nu/2)$ and let $(n_k)_{k\in \N}$ be a sequence in $\Z_{\geq 2}$ such that
$n_{k+1}$ exceeds $n_k$ on a scale $\delta_k$ for each $k$.
Let $\Gamma(k)$ be the collection of all finite paths $\gamma=(\gamma^1,\ldots,\gamma^k)$ where
$\gamma^1\in \mc E(0,n_1)$ and $\gamma^i\in \mc E(n_{i-1},n_{i})$ for $i=2,\ldots,k$ are such that
\eqref{eq: finitary version of realising path} is satisfied.
By Lemma~\ref{lem: finitary argument for interval}, $\Gamma(k)\neq \emptyset$ for all $k\in\N$.
As the $n_k$-head of any element of $\Gamma(K)$ with $K>k$ is an element of $\Gamma(k)$ and further, $\Gamma(k)$ is finite for each $k$, it follows that
\[
 \Gamma(\infty)=\{\gamma\in X_{\mc B}\: \text{ for all } k\in \N \text{, we have } (\gamma_0,\gamma_1,\ldots,\gamma_{n_k})\in \Gamma(k)\}
\]
is non-empty.
Clearly, every $\gamma\in \Gamma(\infty)$ satisfies
\[
 \limsup_{n\to \infty} S^n(\gamma)\geq \nu \qquad \text{ and } \qquad
\liminf_{n\to \infty} S^n(\gamma)\leq \nu.
\]
As $\limsup_{n\to\infty} S^n(\gamma)>0$,
$\gamma$ necessarily traverses vertices with $2$ elements.
With Remark~\ref{rem: full preimages in finite-to-1 extensions}, we see that $\gamma$ is the full preimage
of $\beta=\pi(\gamma)\in X_B$.
Hence, Proposition~\ref{prop: full preimage}  gives that $\beta$ satisfies \eqref{eq: z realises all frequencies}.
This finishes the proof.
\end{proof}

\section{An example with non-maximal Birkhoff spectrum}\label{sec: non-maximal entire section}
In this section, we prove our second main result, Theorem~\ref{thm: non-maximal spectra exist}, which establishes that the assumptions of Theorem~\ref{thm: main} are optimal in some sense.
Specifically, we show that every ordered Bratteli diagram $B$
admits---possibly after telescoping---an extension triple $(B,\hat B,\pi)$ where the $\pi$-fibres have at most $3$ elements
and $0$ is an isolated point of $\SD\neq \{0\}$.
In particular, the Birkhoff spectrum of $(B,\hat B,\pi)$ is not maximal.

\subsection{Colouring $B$}\label{sec: setting up the factor for non-maximal S}
We consider a diagram $B$ which satisfies the standard assumptions from Sections~\ref{sec: preliminaries} and \ref{sec: almost 1 to 1 extensions...} (specifically, $B$ is simple and properly ordered) and utilise the notation introduced there.
In particular, we denote the edges ending in some vertex $v_\ell(n)\in V_n$ by $e_{\ell,m}(n)$ (with appropriate $\ell$ and $m$).
Further, we assume throughout that
\begin{enumerate}[(a)]
\item For each $n\geq 1$, the vertex $v_0(n)$ is the unique source of the minimal edges starting in $V_n$; \label{assum: i}
 \item For $n\geq 2$, each vertex in $V_n$ is the range of at least $5$ edges;\label{assum: ii}
 \item For $n\geq 2$, $v_0(n)$ is connected to each $v\in V_{n-1}$ through at least $4$ edges. \label{assum: iii}
\end{enumerate}
Clearly, \eqref{assum: i} is just a matter of telescoping and, if necessary, relabelling the vertices;
similarly, \eqref{assum: ii} and \eqref{assum: iii} can be ensured through telescoping.
Note that \eqref{assum: i}--\eqref{assum: iii} are unchanged under further telescoping.

In the following, given $n\in \N$,
we \emph{colour} $E_n$ by selecting
integers $1<M_{\ell}(n)<r_\ell(n)-2$ for each vertex $v_\ell\in V_{n+1}$, and labelling
edges $e_{\ell,m}$ in $E_n$ as \emph{thick}
if $m\in \{1,\ldots,M_\ell(n)\}$, and \emph{thin}
if $m\in\{M_\ell(n)+1,\ldots,r_\ell(n)-2\}$.
We call $E_{0,n+1}$ \emph{coloured} if $E_m$ is coloured ($m=1,\ldots,n$) and say $B$ is \emph{coloured} if $E_n$ is coloured for each $n\in \N$.
The reason for calling edges thin and thick, respectively, becomes clear in the next section. 
There, we construct an extension triple $(B,\B,\pi)$ where $B$ is coloured and $\pi$-fibres with more than one element correspond to paths which are thick on almost every level.

Assuming that we have coloured $E_{0,n+1}$,
we write $\thin_n$ for the union of all $n$-cylinders $[\alpha_0,\alpha_1,\ldots,\alpha_n]$
where for some $i\in\{1,\ldots,n\}$, $\alpha_i$ is thin or extremal (maximal or minimal).
We write $\tthin_n$ for the union of $\thin_n$ with all cylinders $[\alpha_0,\alpha_1,\ldots,\alpha_n,\alpha_{n+1}]$ where $\alpha_{n+1}$ is extremal.
Further, we say that
a path $\alpha=(\alpha_k)_{k\geq 0}\in X_B$ \emph{crosses a thin interval in $E_n$ before $t>0$} if
$e_{\ell,0}< \alpha_n\leq e_{\ell,M_\ell(n)}$ and $(\phi^{t'}\alpha)_n=e_{\ell,r_\ell(n)-1}$ for some $t'\in\{1,\ldots,t\}$ and appropriate $\ell\in \{0,\ldots,|V_{n+1}|-1\}$.

As a final piece of terminology, we call
$\gamma=(\gamma_0,\ldots,\gamma_n)\in E_{0,n+1}$ \emph{quite small} if
\[
\gamma_i=e_{0,0}(i) \text{ for } i=0,\ldots, n-2 \quad \text{ and }\quad \gamma_{n-1}=e_{\ell_{n-1},0}(n-1), \ \gamma_n=e_{\ell_{n},1}(n)
\]
for some $\ell_{n-1}\in\{0,\ldots,|V_{n}|-1\}$ and
$\ell_{n}\in\{0,\ldots,|V_{n+1}|-1\}$.
Recalling item \eqref{assum: i} from above, we see that this is equivalent to saying that
$\gamma$ is quite small if it is minimal among those paths in $E_{0,n+1}$ whose edge on level $n$ is not minimal.
\begin{lem}\label{lem: measure of collapsing or extremal sets}
Given a simple properly ordered Bratteli diagram $B'$ satisfying
\eqref{assum: i}--\eqref{assum: iii},
we can telescope $B'$ to obtain a diagram $B$ that satisfies \eqref{assum: i}--\eqref{assum: iii} and can be coloured such that
\begin{enumerate}
\item[(d)] For $n\in \N$, each $\gamma$ whose $n$-head is quite small, and $k\geq 1$,
\begin{align}\label{eq: most edges dont contract two}
1/k\cdot|\{i=0,\ldots,k-1\:(\phi_B^i(\gamma))_n \text{ is not thick}\}|<3^{-n};
\end{align}
 \item[(e)] For $n\geq2$, if $\alpha\in X_{B}$ crosses a thin interval in $E_n$ before $t>0$, then
\begin{align}\label{eq: traversing time is long enough}
1/(t+1)\cdot \sum_{i=0}^{t}\mathbf 1_{\tthin_{n-1}}(\phi^i_{B}(\alpha))<1/2.
\end{align}
\label{item two lem measure of collapsing sets}
\end{enumerate}
\end{lem}

 The proof of the above statement is slightly tedious but, in principle, straightforward.
 Basically, to be able to colour $B$ appropriately we have to telescope $B'$ such that 
 $r_\ell(n)-M_\ell(n)$ is sufficiently large (so $t$ is large enough for the average in \eqref{eq: traversing time is long enough} to be close to the measure of $\tthin_{n-1}$) while still asymptotically negligible when compared to $M_\ell(n)$ (so that \eqref{eq: most edges dont contract two} holds).

\begin{proof}[Proof of Lemma~\ref{lem: measure of collapsing or extremal sets}]
We will recursively define a sequence $0=n_0 <n_1=1<n_2<\ldots$ along which we telescope $B'=(V',E')$ to obtain a diagram $B=(V,E)$
as in the above statement.
Recall that \eqref{assum: i}--\eqref{assum: iii} are preserved under telescoping and thus carry over from $B'$ to $B$.

Before discussing the recursion, we introduce some notation and make a few small observations.
\begin{list}{$\bullet$}{\setlength{\leftmargin}{20pt}} 
 \item
 For $k\geq 1$, we set $i_1^{k}=1+\max_{\gamma\in E_{0,k}}T(\gamma)$, with $T(\gamma)$ as in 
 Definition~\ref{defn: minimal time to maximality},
 and set $i_2^{k}=2i_1^{k}$. 
  Observe that for each $\beta=(\beta_n)_{n\geq 0} \in X_B$, there is $j$ with $0\leq j\leq i_2^{k}$ such that the $(k-1)$-head of $\phi_B^j(\beta)$ is quite small.
 In fact, if $\beta_k$ is maximal, there is $j'$ with $0\leq j'\leq i_2^{k}$ such that the $k$-head of $\phi_B^{j'}(\beta)$ is quite small.
 For our recursion, it is important to note that the definition of $i_1^{k}$ (and $i_2^{k}$) only depends on $E_{0,k}$.
That is to say, if we define $B$ via telescoping $B'$ along $n_0<n_1<\ldots$,
then $i_1^{k}$ is unaffected by the choice of $n_{k+1},n_{k+2},\ldots$
\item Assuming some colouring of $E_n$, every path $\gamma$ whose $n$-head is quite small first visits thick edges $e_{\ell,1},e_{\ell,2},\ldots,e_{\ell,M_\ell(n)}\in E_n$ before visiting thin (or extremal) edges.
Further, as discussed in the previous item, once $(\phi_B^i(\gamma))_n$ is maximal, there is $0\leq j'\leq i_2^{n}$ such that the $n$-head of $\phi_B^{i+j'}(\gamma)$ is again quite small.
As a consequence,
\eqref{eq: most edges dont contract two} readily follows if for quite small $\gamma\in E_{0,n+1}$,
\begin{align}\label{eq: most edges dont contract three}
 \frac{|\{i=0,\ldots,T(\gamma)\:(\phi_B^i(\gamma))_n \text{ is not thick}\}|+ i_2^{n}}{T(\gamma)}
 <3^{-n}/2,
\end{align}
where we chose a smaller right-hand side to minimise technicalities below.

Moreover, if $r(\gamma)=v_\ell$, $|\{i=0,\ldots,T(\gamma)\:(\phi_B^i(\gamma))_n \text{ is not thick}\}|\leq (r_\ell(n)-M_\ell(n))\cdot i_1^{n}$  and $M_\ell(n)\leq T(\gamma)$.
This yields the following sufficient condition for \eqref{eq: most edges dont contract three} (and hence, \eqref{eq: most edges dont contract two})
\begin{align}\label{eq: most edges dont contract four}
\frac{(r_\ell(n)-M_\ell(n))\cdot i_1^{n}+i_2^{n}}{M_\ell(n)}<3^{-n}/2 \quad (\ell=0,\ldots, |V_{n+1}|-1).
\end{align}
\end{list}
 
Starting the recursion, 
we choose $n_2$ big enough to ensure that we can pick $M_\ell(1)$ (for $\ell=0,\ldots,|V_{n_2}'|-1$)
such that \eqref{eq: most edges dont contract four} is satisfied for $n=1$ whenever $B$ is obtained via telescoping $B'$ along a sequence with initial entries $n_0=0$, $n_1=1$ and $n_2$ as chosen.
This implies (d) while (e) is vacuously true for $n=1$.

Now, suppose we have $k\in \N_{\geq 2}$ and $0=n_0 <n_1<n_2<\ldots<n_k$
such that 
whenever $B$ is obtained by telescoping $B'$ along a sequence with initial entries $n_0,n_1,n_2,\ldots,n_k$, we can colour $E_{0,k}$ such that
\eqref{eq: most edges dont contract four} and \eqref{eq: traversing time is long enough} hold for $n=1,\ldots,k-1$.
We fix one such colouring of $E_{0,k}$ and need to determine a level $n_{k+1}$ such that when we telescope $B'$ along a sequence with initial entries $n_0 <n_1<n_2<\ldots<n_k<n_{k+1}$, we can extend the colouring to $E_{0,k+1}$ in a way that additionally ensures \eqref{eq: most edges dont contract four} and \eqref{eq: traversing time is long enough} for $n=k$.

Consider some---at this point not yet fixed---choice of $n_{k+1}>n_k$ and colouring of $E_k$.
Suppose $\alpha\in X_B$ crosses a thin interval in $E_k$ before some $t>0$.
Note that $\frac1{t+1} \sum_{i=0}^t \mathbf 1_{\tthin_{k-1}}(\phi^i_B(\alpha))$ is bounded from above by
\resizebox{\textwidth}{!}{
\begin{math}
\begin{aligned}
 &
 \frac1{t+1} \big[\sum_{i=0}^t \sum_{m=1}^{k-1} \mathbf 1_{\{\gamma\in X_B\:\gamma_m \text{ is not thick}\}}(\phi^i_B(\alpha))+
 \sum_{i=0}^t\mathbf 1_{\{\gamma\in X_B\:\gamma_k \text{ is extremal}\}}(\phi^i_B(\alpha))\big]
 \\&= \sum_{m=1}^{k-1}  \underbrace{\frac{\sum_{i=0}^t \mathbf 1_{\{\gamma\in X_B\:\gamma_m \text{ is not thick}\}}(\phi^i_B(\alpha))}{t+1}}_{A_m}+
 \underbrace{\frac{\sum_{i=0}^t\mathbf 1_{\{\gamma\in X_B\:\gamma_k \text{ is extremal}\}}(\phi^i_B(\alpha))}{t+1}}_{B_k}.
\end{aligned}
\end{math}
}

Towards estimating the averages $A_m$, note that
\eqref{eq: most edges dont contract four} implies $A_m<3^{-m}/2$
if the $m$-head of $\gamma$ is quite small.
While this is not necessarily the case, there is $j$ with $0\leq j\leq i_2^{m+1},t$ such that the
$m$-head of $\phi^j(\alpha)$ is quite small, and thus
\begin{align}\label{eq: estimate Am}
    \frac{\sum_{i=0}^{t-j} \mathbf 1_{\{\gamma\in X_B\:\gamma_m \text{ is not thick}\}}(\phi^{i+j}_B(\alpha))}{t-j+1}<3^{-m}/2.
\end{align}
Set $\rho=\min_{\ell} r_\ell(k)-M_\ell(k)$, where the minimum is taken over all (indices of) vertices $v_\ell\in V_{k+1}=V'_{n_{k+1}}$.
Since $\alpha$ crosses a thin interval in $E_k$ before $t$, we trivially have $\rho\leq t$.
Hence, $j/t\leq i_2^{m+1}/t\leq i_2^{m+1}/\rho$ so that due to \eqref{eq: estimate Am}, we get $A_m<3^{-m}$ ($m=1,\ldots,k-1$)
whenever $ i_2^{m+1}/\rho <3^{-m}/2$.

Towards estimating $B_k$, 
assuming that $n_{k+1}$ is large enough to allow colouring $E_k$ in a way that ensures \eqref{eq: most edges dont contract four} for $n=k$, we have---similar to the case of $A_m$---that $B_k<3^{-k}/2$ as long as the $k$-head of $\alpha$ is quite small.
Again, this is not necessarily the case, so we argue as follows.
Let $t_0\leq t$ mark the first time that $\phi^{t_0}(\alpha)_k$ is maximal (and hence the first time that $\phi^{t_0}(\alpha)_k$ is extremal).
By our initial observations, there is $j$ with $0\leq j\leq i^{k}_2$ such that the $k$-head of $\phi^{t_0+j}(\alpha)$ is quite small.
Using $\rho\leq t$, we obtain
\begin{align*}
 B_k\leq i^{k}_2/\rho+ \frac{\sum_{i=t_0+j}^t\mathbf 1_{\{\gamma\in X_B\:\gamma_k \text{ is extremal}\}}(\phi^i_B(\alpha))}{t+1}<3^{-k},
\end{align*}
where the last inequality holds whenever $i^{k}_2/\rho<  3^{-k}/2$ and \eqref{eq: most edges dont contract four} is satisfied for $n=k$. 
(Note that $\sum_{i=t_0+j}^t\cdots$ is understood to vanish if $t_0+j>t$.)

Clearly, if we choose $n_{k+1}$ sufficiently large, we can
colour $E_k$ in a way that simultaneously ensures
$i^{m}_2/\rho<  3^{-m}/2$ ($m=1,\ldots,k$) and \eqref{eq: most edges dont contract four} (for $n=k$).
By the above, this not only yields \eqref{eq: most edges dont contract four} (and thus \eqref{eq: most edges dont contract two}) but also
\[
 \frac1{t+1} \sum_{i=0}^t \mathbf 1_{\tthin_{k-1}}(\phi^i_B(\alpha))\leq \sum_{m=1}^{k-1} A_m+B_k<\sum_{m=1}^k 3^{-m}<1/2,
\]
that is, \eqref{eq: traversing time is long enough} for $n=k$.
The statement follows.
\end{proof}

\subsection{Extending $B$}\label{sec: construction of non-maximal extension}
In the following, we consider $B$ a simple, properly ordered, coloured diagram satisfying (a)--(e) from Section~\ref{sec: setting up the factor for non-maximal S}.
We construct a diagram $\hat B$ by copy-pasting $B$ in the following way.
\begin{enumerate}[(i)]
 \item \label{item: 3 copies of vertices} For each $n\geq1$, $\V_n$ contains three copies
 $v_\ell^0(n),v_\ell^1(n),v_\ell^2(n)$ of each vertex $v_\ell(n)\in V_n$.
 Only for $\ell=0$, there is an additional fourth copy $v_0^3(n)$ of $v_0(n)$ (which serves as the unique source of the minimal edges in $\E_n$---see the next item).
 \item For each $n\geq 1$ and each $j$ and $\ell$, $s(\e^j_{\ell,0}(n))=v_0^3(n)$ and $s(\e^j_{\ell,r_\ell-1}(n))=v_{\ell_+}^0(n)$
 where $\ell_+=\ell_+(\ell,n)$ is such that $v_{\ell_+}(n)=s(e_{\ell,r_\ell-1}(n))$ in $B$.\label{item: unique source minimal edges}
 \item \label{item: diagram is simple}
 For each $n\geq 1$ and $m=1,\ldots,r_0(n)-2$, $s(\e^3_{0,m}(n))=v_{\ell}^j(n)$ assuming that $s(\e_{0,m}(n))=v_\ell(n)$ (in $B$), where $j\in\{0,1,2\}$ is such that \[j=|\{k=0,\ldots,m-1\:s(e_{0,k}(n))=v_\ell(n)\}|\bmod 3.\]
 Note that with the previous item and assumption \eqref{assum: iii} from the previous section, this ensures that $\B$ is simple.
 \item \label{item: mainly rank is preserved}
 For each $n\geq 1$ and each $\ell$, we have
  \begin{itemize}
  \item For $m=1,\ldots,M_{\ell}(n)$ and $j=0,1,2$, $s(\e^j_{\ell,m}(n))=v_{\ell'}^j(n)$ assuming that $s(\e_{\ell,m}(n))=v_{\ell'}(n)$ (in $B$); 
  \item For $m=M_{\ell}(n)+1,\ldots,r_\ell(n)-2$,
  $s(e_{\ell,m}^0(n))=s(e_{\ell,m}^1(n))=v_{\ell'}^0(n)$
  and $s(e_{\ell,m}^2(n))=v_{\ell'}^1(n)$
  assuming that $s(\e_{\ell,m}(n))=v_{\ell'}(n)$ (in $B$).
 \end{itemize}
  \setcounter{counter}{\value{enumi}}
\end{enumerate}

Observe that the second bullet in item \eqref{item: mainly rank is preserved} implies that if $\alpha\in X_B$ has more than one thin edge, then
$\alpha\notin D$.
On the other hand, the first bullet in that item ensures that if $\alpha$ has no extremal edge on any level $E_n$ with $n\geq 1$ and no more than one thin edge, then $\alpha\in D$.

For the next statement, recall from Section~\ref{sec: irregular extension} that
we call a Bratteli-Vershik system $(X_B,\phi_B)$ $k$-expansive if for each pair of paths $\gamma\neq \gamma'\in X_B$ the $k$-heads of $\phi_B^n(\gamma)$ and $\phi_B^n(\gamma')$ disagree for some $n$.

\begin{thm}\label{thm: non-maximal spectra exist}
 Given an extension triple $(B,\B,\pi)$ with $B$ and $\B$ as described above, $(X_{\B},\phi_{\B})$ is an almost \aab\ extension of $(X_B,\phi_B)$ where $0$ is isolated in $S_D$ and $\SD\neq \{0\}$.
 In particular, $\SD$ is not an interval.
 
 Further, if $(X_B,\phi_B)$ is $k$-expansive,
 then so is $(X_{\B},\phi_{\B})$.
\end{thm}
\begin{proof}
We start by discussing that the extension is at most $3$-to-$1$.
To that end, note that due to items \eqref{item: unique source minimal edges}---\eqref{item: mainly rank is preserved}, the vertex $v_0^3(n)$ ($n\in \N$) is the source
of minimal edges only.
Accordingly, if $\beta\in X_{\B}$ traverses $v_0^3(n)$ on infinitely many levels $n$, then $\beta$ is the sole preimage
of $\pi(\beta)$.
With Remark~\ref{ref: rank is bounded by number of copies} and due to item \eqref{item: 3 copies of vertices}, we hence obtain that fibres have no more than $3$ elements, that is, $|\pi^{-1}\alpha|\leq 3$ for all $\alpha\in X_B$.

We next consider the Birkhoff spectrum $\SD$.
To that end, fix some $\alpha=(\alpha_0,\alpha_1,\ldots)$ in $X_B$ such that $\pi^{-1}(\alpha)\ssq X_{\B}$ is not a singleton.
Note that there is
 $n_0\in \N$ such that $\alpha_n$ is neither extremal nor thin whenever $n\geq n_0$.
 For large enough $M\in \N$, set $\Thin(M)$ to be maximal such that $\alpha$ crosses a thin interval in $E_{\Thin(M)}$ before $M$.
 We assume in the following that $M$ is sufficiently big to ensure $\Thin(M)> n_0$.

As $\alpha$ was arbitrary, to prove that $0$ is isolated in $\SD$, it suffices to show that for every $M$ as above, $S_D^M(\alpha)\geq1/3$.
 To that end, let $N=\Thin(M)$ and let $M_0\in\{1,\ldots, M\}$ be minimal such that $\alpha$ crosses a thin interval in $E_N$ before $M_0$.
 Note that by definition of $N$, $\phi^i_B(\alpha)$ is constant on level $N+2$ and beyond for all $i=0,\ldots,M$,
 that is, $(\phi_B^i(\alpha))_m=\alpha_m$ for $m\geq N+2$.
In particular, we have that $(\phi^i_B(\alpha))_m$ is thick for $i=0,\ldots,M$ and $m\geq N+2$.

Moreover, as $\alpha_{N+1}$ is thick and as $\alpha$ does not cross
a thin interval in $E_{N+1}$ before $M$, $(\phi_B^i(\alpha))_{N+1}$ is not extremal for $i=0,\ldots,M$.
Consequently, if $\phi_B^i(\alpha) \notin \thin_{N}$, then $\phi_B^i(\alpha)\in D$.
Also, observe that for $i=0,\ldots,M_0$, $(\phi_B^i(\alpha))_{N+1}=\alpha_{N+1}$ is thick 
so that for $i=0,\ldots,M_0-1$, $\phi^i_B(\alpha)\in D$ already if $\phi_B^i(\alpha) \notin \thin_{N-1}$.
It follows that
if $\tau\geq0$ is minimal such that the $N$-head of
$\phi^\tau_B(\alpha)$ is quite small, then as long as $0\leq i\leq \tau$,
$\phi^i_B(\alpha)\in D$ whenever $\phi_B^i(\alpha) \notin \tthin_{N-1}$.

 Altogether, we obtain
 \begin{align*}
  S_D^M(\alpha)
  &\geq 1-1/M\cdot |\{i=0,\ldots,M-1\: \phi^i(\alpha) \in \tthin_{N-1}\}|\\
  &\phantom{\geq\ }-1/M\cdot |\{i=\tau,\ldots,M-1\: (\phi^i(\alpha))_N \text{ is not thick}\}|
  \\
  &
  \geq1- 1/2- 3^{-N} \geq 1/3,
 \end{align*}
where we used \eqref{eq: traversing time is long enough} and \eqref{eq: most edges dont contract two} in the penultimate step.
This shows the first part of the statement.

Now, suppose $(X_B,\phi_B)$ is $k$-expansive.
As in the proof of Theorem~\ref{thm: 2 to 1 extensions exist}, it suffices to show that
given $\beta\neq\beta'\in X_{\B}$ with $\alpha=\pi(\beta)=\pi(\beta')$, 
there is $i\in\Z$ with $(\phi_{\B}^i(\beta))_n\neq (\phi_{\B}^i(\beta'))_n$ for each $n\geq 0$.
To that end, pick $n_0$ as above (such that $\alpha_n$ is thick for all $n\geq n_0$) and
choose a path $\gamma=(\gamma_k)_{k=0,\ldots,n_0-1}\in E(v_0,s(\alpha_{n_0}))$ such that
$\gamma_k$ is thick ($k=1,\ldots,n_0-1$).
Observe that there is $i$ such that 
the $(n_0-1)$-head of $\phi^i_B(\alpha)$ coincides with $\gamma$ and
$\phi_B^i(\alpha)_n=\alpha_n$ for all $n\geq n_0$.
In particular, $(\phi_B^i(\alpha))_n$ is thick for each $n\geq 1$ so that $\phi_B^i(\alpha)\in D$.
The statement follows.
\end{proof}

\section{Application to symbolic almost automorphic extensions and  Birkhoff spectra of fat Cantor sets}\label{sec: last section}
In the following, $\T$ is an infinite compact metrisable monothetic group (that is, $\T$ has a dense cyclic subgroup) equipped with normalised Haar measure $m$.
For concreteness, we fix some compatible metric $d$ on $\T$.
By $(\T,\g)$, we denote the \emph{rotation} given by $\T\ni\theta\mapsto \theta+\g\in \T$.
Our goal is to apply the results of the previous sections towards computing the Birkhoff spectra $\SC$ of suitable subsets $C\ssq \T$ under the rotation by $\g$, that is, we want to compute
\[
 \SC=\bigcup_{\theta\in \T}\big\{\nu\in[0,1]\:\nu \text{ is an accumulation point of } 1/n\cdot\sum_{i=0}^{n-1}\mathbf 1_{C}(\theta+i\g)\big\}.
\]
The \emph{suitable} sets whose Birkhoff spectra our results apply to, in principle, are the boundaries of certain covers of $\T$, described in Section~\ref{sec: aa subshifts} below.

\subsection{Technical preparation}
We begin with a brief discussion of natural isomorphisms between expansive Bratteli–Vershik systems and corresponding subshifts. 
We then introduce the basics of almost automorphic subshifts along with their associated separating covers.
The boundaries of these covers are the sets whose Birkhoff spectra we are able to analyse.

\subsubsection{Expansive Bratteli-Vershik systems and subshifts}
\label{sec: expansive BV systems and their subshifts}
Recall that $(X,f)$ is an expansive Cantor system if and only if it is isomorphic to a bi-infinite subshift $(\Sigma,\sigma)$ over a finite alphabet $\mc A$ \cite{Hedlund1969}.
If $(X,f)$ is given by a Bratteli-Vershik system $(X_B,\phi)$, a corresponding subshift and isomorphism can be identified as follows.

Assuming that $(X_B,\phi)$ is $0$-expansive (see Section~\ref{sec: irregular extension}), let $\mc A=E_0$ (the level-$0$ edges) and consider
the collection $\Sigma$ of all sequences $x=(x_i)\in {\mc A}^\Z$ 
over $\mc A$
such that there is $\gamma\in X_B$ with
\[x_i=\phi^i(\gamma)_0 \qquad (i\in \Z).\]
Then, $\Sigma$ is straightforwardly seen to be closed (in the product topology on ${\mc A}^\Z$) and shift invariant, that is, $(\Sigma,\sigma)$ is a subshift.
Moreover,
\begin{align*}
 h\: X_B\to \Sigma,\qquad \gamma\mapsto \big(\phi^i(\gamma)_0\big)_{i\in \Z}
\end{align*}
is an isomorphism between $(X_B,\phi)$ and $(\Sigma,\sigma)$, see also \cite{GjerdeJohansen2000,DownarowiczMaass2008}.

\subsubsection{Almost automorphic subshifts}\label{sec: aa subshifts}
In all of the following, $(\T,\g)$ is a \emph{minimal} rotation and hence, uniquely ergodic (with $m$ being the unique invariant measure).
A topological dynamical system $(X,f)$ is an \emph{(irregular/ regular) almost automorphic system}
over $(\T,\g)$
if $(X,f)$ is an (irregular/ regular) almost $1$-to-$1$ extension of $(\T,\g)$.
Clearly, being regularly (or irregularly) almost automorphic over a specific rotation is preserved under isomorphisms.
It is important to note that regularity and irregularity are mutually exclusive in the context of almost automorphy.
Indeed, an almost automorphic system over $(\T,\g)$ is either a regular or an irregular extension of $(\T,\g)$.

A finite collection $W=(W_0,W_1,\ldots,W_{m-1})$ of subsets of $\T$ is
a \emph{topologically regular cover} if
\begin{enumerate}[(a)]
 \item $\T = \bigcup_{i=0,\ldots,m-1} W_i$ and $\operatorname{int} W_i\cap \operatorname{int} W_j=\emptyset$ whenever $i\neq j$;
 \item For all $i=0,\ldots,m-1$, $W_i$ is \emph{topologically regular}, that is, $\overline{\operatorname{int} W_i}=W_i$.
 \end{enumerate}
Given such $W$, we set $\partial W=\bigcup_{i=0,\ldots,m-1}\partial W_i$.
 Following \cite{Paul1976,MarkleyPaul1979}, we call a topologically regular cover $W$ a \emph{separating cover} if
 \begin{enumerate}[(a)]
 \setcounter{enumi}{2}
 \item For all $\theta\neq \theta'\in \T$, there are $i\neq j$ and some $n\in \Z$ with $\theta+ng \in \operatorname{int} W_i$ and $\theta'+ng \in \operatorname{int} W_j$.
\end{enumerate}

Key to our analysis is a close relation between separating covers and almost automorphic subshifts.
Given a separating cover $W=(W_0,W_1,\ldots,W_{m-1})$ of $\T$,
there is a canonically associated almost automorphic subshift $(\S_W,\s)$ over $(\T,\g)$ with an almost $1$-to-$1$ factor map $\pi\:(\Sigma_W,\sigma)\to (\T,\g)$
such that
\begin{align}\label{eq: cover associated to aa subshift}
W_i=\{\theta\in \T\: \text{there is } x\in \pi^{-1}\theta \text{ with } x_0= i\},
 \end{align}
see \cite[Theorem~2.5]{Paul1976}.
As a consequence, we readily obtain
\begin{align}\label{eq: boundary different entries}
\d W=\{\theta\in \T\: \text{there are } x,y\in \pi^{-1}\theta \text{ with } x_0\neq y_0\}
\end{align}
as well as
\begin{align}\label{eq: coding in almost automorphic shifts}
x_n =i\, \text{ if }\, \pi(x)+ng \in \operatorname{int} W_i \qquad \text{and} \qquad
 \pi(x)+ng \in W_i\,  \text{ if }\, x_n =i
\end{align}
for all $x=(x_n)_{n\in \Z}\in \S_W$ and $n\in \Z$.

The converse is also true \cite[Theorem~2.6]{Paul1976}:
Given an almost $1$-to-$1$ factor map $\pi\:(\Sigma,\sigma)\to (\T,\g)$, there is a separating cover $W=(W_0,\ldots,W_{m-1})$ of $\T$ with $\Sigma=\Sigma_W$ and such that \eqref{eq: cover associated to aa subshift}--\eqref{eq: coding in almost automorphic shifts} are satisfied---note that here, we assume without loss of generality that $(\S,\s)$ is a subshift over the alphabet $\{0,\ldots,m-1\}$ for some $m\in \N$.
Observe that \eqref{eq: boundary different entries} implies that $(\Sigma_W,\s)$ is an irregular almost $1$-to-$1$ extension of $(\T,g)$ if and only if $m(\d W)>0$.
\subsection{Birkhoff spectra of boundaries of separating covers}
We are now in a position to translate Theorems~\ref{thm: main} and \ref{thm: non-maximal spectra exist} into analogous statements concerning the asymptotic frequency of visits to the boundaries of separating covers. 
Although the resulting formulations are similar, some caution is necessary. 
In particular, when recasting the problem as a computation of the Birkhoff spectrum of an extension triple, one must account for the fact that the spectrum crucially depends on the particular triple, see the discussion in Section~\ref{sec: collapsing map}. 
As a consequence, and especially in the proof of the second part of the following statement, we have to choose (the first levels of) our Bratteli–Vershik representations carefully.

\begin{thm}\label{thm: translation to almost automorphic shifts}
 Given a minimal rotation $(\T,\g)$, there are separating covers $W,W'$ of $\T$ with $m(\d W),m(\d W')>0$ such that
 $\SW$ is maximal (that is, $\SW=[0,m(\d W)]$) and $\SWP$ is not maximal.
 
 In fact, whenever $(\Sigma_W,\s)$ is an (irregular) almost \aaa\ extension of $(\T,\g)$, then $\SW$ is maximal.
\end{thm}
It is a classical fact that over any minimal rotation $(\T,\g)$ there is an almost automorphic subshift $(\Sigma,\sigma)$ \cite[Theorem~3.1]{Paul1976}.
Below, we will need a refinement of this statement (see \cite[Corollary~3.13]{FuhrmannKwietniak2020} and its proof):
To each element $\theta_0\in \T$ and $\eps>0$,
there is a regular almost automorphic subshift over the alphabet $\{0,1\}$ where the associated cover $W=(W_0,W_1)$ is such that $W_1$
is entirely contained in the $\eps$-ball $B_{\eps}(\theta_0)$ (with respect to the metric $d$ on $\T$) around $\theta_0$.
\begin{proof}[Proof of Theorem~\ref{thm: translation to almost automorphic shifts}]
  Pick some regular almost automorphic subshift $(\S,\s)$
  over $(\T,\g)$ with Bratteli-Vershik representation $(X_B,\phi_B)$.
  Towards a cover $\hat W$ with $\SWH$ maximal,
 we use Theorem~\ref{thm: 2 to 1 extensions exist} to obtain (possibly after telescoping) an extension triple $(B,\B,q)$, where
 $(X_{\B},\phi_{\B})$ is isomorphic to a
 subshift $(\hat \Sigma,\sigma)$ and
 $q$ is an irregular almost \aaa\ factor map.
  Let $D$ be as in \eqref{eq: defn discontinuities} (for $(B,\B,\q)$).

Without loss of generality, we may assume that $(X_{B},\phi_{B})$ and
$(X_{\B},\phi_{\B})$ are $0$-expansive and that
we are given isomorphisms $h\:(X_B,\phi_B)\to(\Sigma,\sigma)$ and $\hat h\: (X_{\B},\phi_{\B})\to (\hat \Sigma,\sigma)$ as in Section~\ref{sec: expansive BV systems and their subshifts}.
In particular, this implies $\mu(D)>0$ for each invariant measure $\mu$ of $(X_B,\phi_B)$ (see Section~\ref{sec: collapsing map}) so that $S_D=[0,\sup_\mu \mu(D)]$ is a non-degenerate interval, see Theorem~\ref{thm: main}.
  Set $q'=h\circ q\circ {\hat h}^{-1}$ (so, $q'\: (\hat\S,\s)\to(\S,\s)$) and
let $\pi\: (\Sigma,\sigma)\to(\T,\g)$ be a (necessarily regular) almost $1$-to-$1$ factor map.

Let $\hat\pi\:(\hat\Sigma,\sigma)\to(\T,\g)$ be given by $\hat \pi=\pi\circ q'$.
Recall that the collection of regular $\pi$-fibres is residual in $X_B$, and so is the projection (under $q'$) of the regular $q'$-fibres.
As a consequence, $\hat\pi$ is an almost $1$-to-$1$ factor map.
Denote by $W$ and $\hat W$ the 
separating covers of $\T$ corresponding to $\pi\:(\Sigma,\s)\to (\T,g)$ and $\hat\pi\:(\hat \Sigma,\s)\to (\T,g)$, respectively, as described in \eqref{eq: cover associated to aa subshift}.

Note that $\mathbf 1_{h(D)}(\s^i x)\leq \mathbf 1_{\d \hat W}(\theta+i\g)\leq \mathbf 1_{h(D)\cup \pi^{-1}(\d W)}(\s^i x)$ for each $\theta\in \T$, $x\in \pi^{-1}(\theta)$, and $i\in \Z$.
This is a consequence of \eqref{eq: boundary different entries}, the definition of $D$, and the definition of $\hat h$.
As $m(\d W)=0$ (since $(\S,\s)$ is a regular extension), we have
$ 1/n\cdot \sum_{i=0}^{n-1}\mathbf 1_{\pi^{-1}(\d W)}(\s^i x)\stackrel{n\to \infty}{\longrightarrow} 0$---see Proposition~\ref{prop: limsup is assumed as limit} and note that every invariant measure of an almost automorphic extension of $(\T,g)$ necessarily projects to $m$.
Accordingly, as $n\to\infty$, we asymptotically have
\begin{align}\label{eq: computing asymptotic frequencies}
\begin{split}
 1/n\cdot \sum_{i=0}^{n-1}\mathbf 1_{\d \hat W}(\theta+i\g)\sim
 1/n\cdot \sum_{i=0}^{n-1}\mathbf 1_{h(D)}(\s^i x)
 =1/n\cdot \sum_{i=0}^{n-1}\mathbf 1_{D}(\phi_B^i\gamma),
 \end{split}
\end{align}
where $\gamma=h^{-1}(x)$.
It follows that $\SWH=S_D$ is maximal.
Using extensions as in Theorem~\ref{thm: non-maximal spectra exist} in place of Theorem~\ref{thm: 2 to 1 extensions exist}, the existence of non-maximal spectrum is proven similarly.

For the ``in fact''-part, consider $(\S_W,\s)$ as in the statement, let $W=(W_0,\ldots,W_{m-1})$ be a corresponding separating cover, and pick $\theta_0\in \T$ and $\eps>0$ such that
$B_\eps(\theta_0)\ssq \operatorname{int} W_0$.
Let $(\S_{W'},\s)$ be a regularly almost automorphic subshift over $(\T,g)$ with a separating cover
$W'=(W_0',W'_1)$ such that $W_1'\ssq B_\eps(\theta_0)$ (see the paragraph preceding the present proof).
Then $\hat W=(\hat W_0,\ldots,\hat W_m)=(W_0\setminus \operatorname{int}W_1',W_1,\ldots,W_{m-1},W_1')$ is readily seen to be a separating cover; we refer by $(\S_{\hat W},\s)$ to the corresponding almost automorphic subshift over $(\T,g)$.
Let $\pi_{\hat W}\: (\Sigma_{\hat W},\s)\to (\T,\g)$ and $\pi_{W'}\:(\S_{W'},\s)\to(\T,\g)$ be the canonical almost $1$-to-$1$ factor maps.

Note that we have an almost $1$-to-$1$ factor map $q'\: (\S_{\hat W},\s)\to( \S_{W'},\s)$ given by
\[
q'((x_n)_{n\in\Z})_i=
 \begin{cases}
  1 &\text{if } x_i=m,\\
  0 &\text{otherwise}
 \end{cases}\qquad(i\in \Z).
\]
Indeed, $q'$ is obviously continuous, commutes with the shift, and thus maps orbit closures to orbit closures.
Further, if $x$ is a regular $\pi_{\hat W}$-fibre, then $\theta=\pi_{\hat W}(x)$ satisfies $\theta+\Z g\cap \d W'\ssq \theta+\Z g\cap \d \hat W=\emptyset$.
Hence, using \eqref{eq: coding in almost automorphic shifts}, we see that $q'(x)\in \S_{W'}$ and $\pi_{W'}\circ q'(x)=\pi_{\hat W}(x)$.
Both facts immediately extend to the entire orbit of $x$
and hence, to all of $\S_{\hat W}$ by  minimality and continuity.
Altogether, we conclude that $q'$ is a well-defined (that is, $q'(\S_{\hat W})=\S_{W'}$) almost $1$-to-$1$ factor map with $\pi_{W'}\circ q'=\pi_{\hat W}$.

In fact, $q'$ is almost \aaa.
To see this, note that similar to the definition of $q'$,
\begin{align}
 \label{eq: defn q pime prime}
q''((x_n)_{n\in\Z})_i=
 \begin{cases}
  0 &\text{if } x_i=m,\\
  x_i &\text{otherwise}
 \end{cases}\qquad(i\in \Z)
\end{align}
defines a factor map $q''\: (\S_{\hat W},\s)\to (\S_{W},\s)$ with
 $\pi_{W}\circ q''=\pi_{\hat W}$.
In particular, 
\begin{align}\label{eq: different paths of factoring}
\pi_{W}\circ q''=\pi_{W'}\circ q'.
\end{align}
Now, assuming for a contradiction that there are distinct $x,y,z\in \S_{\hat W}$ with 
$q'(x)=q'(y)=q'(z)$, there must be $n_1,n_2,n_3\in \Z$ with
$x_{n_i},y_{n_i},z_{n_i}\neq m$ (for $i=1,2,3$) and $x_{n_1}\neq y_{n_1}$, 
$x_{n_2}\neq z_{n_2}$, and $y_{n_3}\neq z_{n_3}$.
With \eqref{eq: defn q pime prime}, we see that $q''(x)$, $q''(y)$ and $q''(z)$ would then be three distinct points in $\S_{W}$ which, due to \eqref{eq: different paths of factoring}, are identified under $\pi_W$, contradicting the assumption of $\pi_W$ being almost \aaa.

Finally, let $P=\{[i]\ssq \Sigma_{\hat W}\:i=0,\ldots,m\}$ and $Q=\{[i]\ssq \Sigma_{W'}\:i=0,1\}$ be
the partitions of $\Sigma_{\hat W}$ and $\Sigma_{W'}$, respectively, by cylinder sets of length $1$.
Clearly, ${q'}^{-1}(Q)$ is coarser than $P$.
As discussed in Remark~\ref{rem: adapted to a partition}, there is hence an extension triple $(B,\hat B,q)$ where $(X_B,h)$ and $(X_{\hat B},\hat h)$ are Bratteli-Vershik representations of
$(\S_{W'},\s)$ and $(\S_{\hat W},\s)$, respectively,
such that $(X_{\hat B},\hat h)$ is adapted to $P$ and $(X_{B},h)$ is 
adapted to $Q$.
In particular, $(X_{\hat B},\phi_{\B})$ and $(X_{B},\phi_{B})$ are $0$-expansive, and
it is easy to see that $h$ and $\hat h$ coincide, up to relabelling, with the respective isomorphisms from Section~\ref{sec: expansive BV systems and their subshifts}.
 As further, $q'=h\circ q\circ {\hat h}^{-1}$,
 we are in a similar situation as before and the statement follows as in the first part (with $\Sigma_W'$ in place of $\S$ and $\S_{\hat W}$ in place of $\hat\S$).
\end{proof}
\begin{proof}[Proof of Corollary~\ref{cor: birkhoff spectra of irrational rotations introduction}]
 Similar to the first part of the proof of Theorem~\ref{thm: translation to almost automorphic shifts}, we can arrange for a situation where we are given an extension triple $(B,\B,q)$---obtained via Theorem~\ref{thm: 2 to 1 extensions exist} and Theorem~\ref{thm: non-maximal spectra exist} for non-degenerate maximal and non-maximal spectra, respectively---such that $(B,h)$ and $(\B,\hat h)$ are $0$-expansive Bratteli-Vershik representations of almost automorphic subshifts $(\S,\s)$ and $(\hat\S,\s)$
  over $(\T^1,\g)$ with almost $1$-to-$1$ factor maps $\hat \pi\:(\hat \S,\s)\to(\T^1,\g)$, $\pi\:(\S,\s)\to(\T^1,\g)$ such that $\hat \pi=\pi\circ q'$, where $q'=h\circ q \circ {\hat h}^{-1}$ and $h$ and $\hat h$ are as in Section~\ref{sec: expansive BV systems and their subshifts}.
 
 For simplicity, we may choose $(\Sigma,\s)$ to be a Sturmian subshift with irregular $\pi$-fibres---each consisting of exactly two elements---over the orbit $\Z\w\in \T^1$ of $0$.
 Moreover, we may arrange for $(B,\B,q)$ to be such that $h^{-1}(\pi^{-1}(0))$ contains the minimal path in $B$ (and hence, the projection of a regular $q$-fibre).
 
 As in the proof of Theorem~\ref{thm: translation to almost automorphic shifts}, we obtain $\SWH=S_D$ with $D$ as in \eqref{eq: defn discontinuities}.
 Clearly, $\mathbf 1_{D}(\phi_B^i\gamma)=
 \mathbf 1_{h(D)}(\s^ix)$ for each $\gamma\in X_B$ and $i\in \Z$, where $x=h(\gamma)$.
 Moreover, if $x$ is a regular $\pi$-fibre, we also have
 $\mathbf 1_{h(D)}(\s^ix)=
 \mathbf 1_{\pi(h(D))}(\pi(x)+i\w)$.
 If, however, $x$ lies in an irregular fibre $\{x,y\}$, we may assume without loss of generality that
 $x$ corresponds to a regular $q$-fibre (since $h^{-1}(\pi^{-1}(0))$ contains the minimal path) and hence, never visits $h(D)$ so that $\mathbf 1_{\pi(h(D))}(\pi(x)+i\w)= \mathbf 1_{\pi(h(D))}(\pi(y)+i\w)=\mathbf 1_{h(D)}(\s^iy)$.
 Altogether, this shows $\SWH=\SWPI$.

 The statement hence follows if we can show that $\pi(h(D))$ is a Cantor set---note that fatness, that is, $m(\pi(h(D)))>0$ follows from the non-singleton Birkhoff spectra and Proposition~\ref{prop: limsup is assumed as limit}.

As $D$ is non-empty and compact,
 so is $\pi(h(D))$.
Clearly, on the circle, nowhere dense implies totally disconnected.
Hence, since $\d \hat W\supseteq \pi(h(D))$, $\pi(h(D))$ is totally disconnected.
It remains to show that
 $\pi(h(D))$ has no isolated points.
 
 To that end, notice that the constructions in Theorem~\ref{thm: 2 to 1 extensions exist} and Theorem~\ref{thm: non-maximal spectra exist} are such that $D$ (and hence, $h(D)$) has no isolated points.
 Now, if, for a contradiction, $\pi(h(D))$ had isolated points, then continuity of $\pi$ would imply that there is an open set $U\ssq \S_W$ with $\pi(U\cap h(D))$ a singleton.
 Consequently, $\pi$ would identify infinitely many points contradicting our assumptions on $(\S,\s)$.
\end{proof}
\begin{rem}
 The Cantor sets in the above proof are given rather implicitly.
 We refer the reader to the appendix of \cite{FuhrmannGlasnerJagerOertel2021} for an explicit construction of so-called \emph{perfectly self-similar} fat Cantor sets $C$ in $\T^1$.
 Such Cantor sets can constitute the boundary $\d W$ of a cover whose associated almost automorphic subshift $(\S_W,\s)$ is an irregular almost \aaa\ extension of $(\T^1,\w)$. 
 With the second part of Theorem~\ref{thm: translation to almost automorphic shifts}, we see that $S_C$ is maximal for such $C$.
\end{rem}

\bibliography{dynamics}
\bibliographystyle{abbrv}
\end{document}